\RequirePackage{snapshot}

\newif\ifsubsections
\subsectionstrue 

\documentclass[]{pcmi}

\usepackage[sc]{mathpazo}          
\usepackage{eulervm}               
\usepackage[scaled=0.86]{berasans} 
\usepackage[scaled=1]{inconsolata} 
\usepackage[T1]{fontenc}

\usepackage[%
	protrusion=true,
	expansion=false,
	auto=false
	]{microtype}

\usepackage{xcolor}

	\definecolor{linkred}{rgb}{0.0,0.0,0.0}
	\definecolor{linkblue}{rgb}{0,0.0,0.0}

\usepackage{amsmath, amssymb}

\PassOptionsToPackage{hyphens}{url} 
\usepackage[
    setpagesize=false,
    pagebackref,
	pdfpagelabels=false,
    pdfstartview={FitH 1000},
    bookmarksnumbered=false,
    linktoc=all,
    colorlinks=true,
    anchorcolor=black,
    menucolor=black,
    runcolor=black,
    filecolor=black,
    linkcolor=linkblue,
	citecolor=linkblue,
	urlcolor=linkred,
]{hyperref}
\usepackage[backrefs,msc-links,nobysame]{amsrefs}

\customizeamsrefs 

\theoremstyle{plain}
\newtheorem{theorem}[equation]{Theorem}
\newtheorem{proposition}[equation]{Proposition}

\newtheorem{lemma}[equation]{Lemma}

\theoremstyle{definition}
\newtheorem{remark}[equation]{Remark}
\newtheorem{definition}[equation]{Definition}
\newtheorem{conjecture}[equation]{Conjecture}
\newtheorem{assumption}[equation]{Assumption}

\newcommand{\N} {\mathbb{N}}
\newcommand{\C} {\mathbb{C}}
\newcommand{\R} {\mathbb{R}}

\newcommand{\E} {\mathbb{E}}
\renewcommand{\P} {\mathbb{P} \, }

\newcommand{\BB} {\mathcal{B}}
\newcommand{\DD} {\mathcal{D}}
\newcommand{\LL} {\mathcal{L}}
\newcommand{\NN} {\mathcal{N}}

\renewcommand{\a} {\alpha}
\newcommand{\e} {\varepsilon}
\renewcommand{\d} {\delta}
\newcommand{\D} {\Delta}
\newcommand{ \Om} {\Omega}
\renewcommand{\t} {\tau}
\renewcommand{\l} {\lambda}

\renewcommand{\Pr}[2][]{\mathbb{P}_{#1} \left\{ #2 \rule{0mm}{3mm}\right\}}
\newcommand{\etc} {,\ldots,}

\DeclareMathOperator{\dist }{dist}
\DeclareMathOperator{\Span }{span}

\DeclareMathOperator{\real}{Real}

\newcommand{\pr}[2]{\left \langle {#1} , {#2} \right \rangle}
\newcommand{\norm}[1]{\left \| #1 \right \|}

\begin{document}

%
%
%
%
%
%

\title{Delocalization of eigenvectors of random matrices \\ Lecture notes}

%
%
\author{Mark Rudelson}\thanks{ Partially supported by NSF grant, DMS-1464514.}
\address{ Department of Mathematics, University of Michigan. }
\email{rudelson@umich.edu}
%
%

\begin{abstract}
Let $x \in S^{n-1}$ be a unit eigenvector of an $n \times n$ random matrix. This vector is delocalized if it is distributed roughly uniformly over the real or complex sphere. This intuitive notion can be quantified in various ways. In these lectures, we will concentrate on the \emph{no-gaps delocalization}. This type of delocalization means that with high probability, any non-negligible subset of the support of $x$ carries a non-negligible mass.  Proving the no-gaps delocalization requires establishing small ball probability bounds for the projections of random vector. Using Fourier transform, we will prove such bounds in a simpler case of a random vector having independent coordinates of a bounded density. This will allow us to derive the no-gaps delocalization for matrices with random entries having a bounded density. In the last section, we will discuss the applications of delocalization to the spectral properties of Erd\H{o}s-R\'enyi random graphs.
\end{abstract}  

%
%
\maketitle
\thispagestyle{empty}

%
%


\section{introduction}

Let $G$ be a symmetric random matrix with independent above the diagonal normal random entries having expectation $0$ and variance $1$ ($N(0,1)$ random variables). The distribution of such matrices is invariant under the action of the orthogonal group $O(n)$. Consider a unit eigenvector $v \in S^{n-1}$ of this matrix. The distribution of the eigenvector should share the invariance of the distribution of the matrix itself, so $v$ is uniformly distributed over the real unit sphere $S_{\R}^{n-1}$. Similarly, if $\Gamma$ is an $n \times n$ complex random matrix with independent entries whose real and imaginary part are independent $N(0,1)$ random variables, then the distribution of $\Gamma$ is invariant under the action of the unitary group $U(n)$. This means that any unit eigenvector of $\Gamma$ is uniformly distributed over the complex unit sphere $S_{\C}^{n-1}$.  For a general distribution of entries, we cannot expect such strong invariance properties. Indeed, if the entries of the matrix are random variables taking finitely many values, the eigenvectors will take finitely many values as well, so the invariance is impossible.
Nevertheless, as $n$ increases, a central limit phenomenon should kick in, so the distribution of an eigenvector should be approximately uniform. This vague idea called delocalization can be made mathematically precise in a number of ways. Some of these formalizations use the local structure of a vector. One can fix in advance several coordinates of the eigenvector and show that the joint distribution of  these coordinates approaches the distribution of a properly normalized gaussian vector, see \cite{BY}.

In these notes, we adopt a different approach to delocalization coming from the non-asymptotic random matrix theory. The asymptotic theory is concerned with establishing limit distributions of various spectral characteristics of a family of random matrices when the sizes of these matrices tend to infinity. In contrast to it, the non-asymptotic theory strives to obtain explicit, valid with high probability bounds for the matrices of a large fixed size. This approach is motivated by applications primarily to convex geometry, combinatorics, and computer science. For example, while analyzing performance of an algorithm solving a noisy linear system, one cannot let the size of the system go to infinity. 
An interested reader can find an introduction to the non-asymptotic theory in \cite{RV ICM, V survey, R survey}.
In this type of problems, strong probabilistic guarantees are highly desirable, since one typically wants to show that many ``good'' events occur at the same time. This will be the case in our analysis of the delocalization behavior as well

 We will  consider the global structure of the eigenvector of a random matrix controlling all coordinates of it at once. The most classical type of such delocalization is the $\ell_{\infty}$ norm bound. If $v \in S^{n-1}$ is a random vector uniformly distributed over the unit sphere, then with high probability, all its coordinates are small. This is easy to check using the concentration of measure. Indeed, the vector $v$ has the same distribution as $g/\norm{g}_2$, where $g \in \R^n$ or $\C^n$ is the standard Gaussian vector, i.e., a vector with the independent $N(0,1)$ coordinates. By the concentration of measure, $\norm{g}_2=c \sqrt{n}(1+o(1))$ with high probability. Also, since the coordinates of $g$ are independent,
 \[
   \E \norm{g}_{\infty}= \E \max_{j \in [n]} |g_j| \le C \sqrt{\log n},
 \]
 and the measure concentration yields that $\norm{g}_{\infty} \le C' \sqrt{\log n}$ with high probability. Therefore, with high probability,
 \[
   \norm{v}_{\infty} \le C \frac{\sqrt{\log n}}{\sqrt{n}}.
 \]
 Here and below, $C, \bar{C},C', c$, etc. denote absolute constants which can change from line to line, or even within the same line.
 
 One would expect to have a similar $\ell_{\infty}$ delocalization for a general random matrix. The bound
 \[
   \norm{v}_{\infty} \le C \frac{\log^c n}{\sqrt{n}}
 \]
 for unit eigenvectors was proved in \cite{ESY 1, ESY 2} for Hermitian random matrices and in \cite{RV delocalization} for random matrices all whose entries are independent. Moreover, in the case of the Hermitian random matrix with i.i.d. subgaussian entries, the previous estimate has been established with the optimal power of the logarithm $c=1/2$, see \cite{VW}.
 We will not discuss the detailed history and the  methods of obtaining the $\ell_{\infty}$ delocalization in these notes, and refer a reader to a comprehensive recent survey \cite{OVW}.

 Instead, we are going to concentrate on a different manifestation of the delocalization phenomenon. The $\ell_{\infty}$ delocalization rules out peaks in the distribution of mass among the coordinates of a unit eigenvector.
  In particular, it means that with high probability, the most of the mass, i.e., $\ell_2$ norm of a  unit eigenvector cannot be localized on a few coordinates. We will consider a complementary phenomenon, namely ruling out chasms in the mass distribution. More precisely, we aim at showing that with high probability, any non-negligible set of the coordinates of a unit eigenvector carries a relatively large mass. We call this property of lack of almost empty zones in the support of the eigenvector the \emph{no-gaps delocalization}.

No-gaps delocalization property holds for the eigenvectors of many natural classes of random matrices. This includes matrices, whose all entries are independent, random real symmetric and skew-symmetric matrices, random complex hermitian matrices with independent real and imaginary parts of the entries, etc. We formulate the explicit assumption on the dependencies of the entries below.
\begin{assumption}[Dependencies of entries]         \label{A}
  Let $A$ be an $n \times n$ random matrix.
  Assume that for any $i,j \in [n]$, the entry $A_{ij}$ is independent of the rest of the entries except possibly $A_{ji}$.
We also assume that the real part of $A$ is random and the imaginary part is fixed.
\end{assumption}

Fixing the imaginary part in Assumption~\ref{A} allows us to handle real random matrices.
This assumption can also be arranged for complex matrices with independent real and imaginary parts,
once we condition on the imaginary part. One can even consider a more general situation
where the real parts of the entries conditioned on the imaginary parts have variances bounded below.

We will also assume $\|A\| = O(\sqrt{n})$ with high probability.
This natural condition holds, in particular, if the entries of $A$ have mean zero and
bounded fourth moments (see, e.g., \cite{Tao book}).
To make this rigorous, we fix a number $M \ge 1$
and introduce the boundedness event
\begin{equation}         \label{eq: BAM}
\BB_{A,M} := \left\{ \|A\| \le M \sqrt{n} \right\}.
\end{equation}

We will formulate two versions of  the no-gaps delocalization theorem, for absolutely continuous entries with bounded density and for general entries. Although the second case is includes the first one, the results under  the bounded density assumtion are stronger, and the proofs are significantly easier. Let us formulate the first assumption explicitly.

\begin{assumption}[Continuous distributions]	\label{A: continuous distribution}
  We assume that the real parts of the matrix entries have
  densities bounded by some number $K \ge 1$.
\end{assumption}

Under Assumptions \ref{A} and \ref{A: continuous distribution}, we show that every subset of at least
eight coordinates carries a non-negligible part of the mass of any eigenvector.
This is summarized in the following theorem.

\begin{theorem}[Delocalization: continuous distributions]	\label{thm: delocalization continuous}
  Let $A$ be an $n \times n$ random matrix which satisfies Assumptions \ref{A} and \ref{A: continuous distribution}.
  Choose $M \ge 1$.
  Let $\e \in [8/n,1)$ and $s>0$.
  Then,  the following event holds with probability
  at least
  \[
    1-(Cs)^{ \e n}- \P(\BB_{A,M}^c).
  \]
  Every eigenvector $v$ of $A$ satisfies
  $$
  \|v_I\|_2 \ge (\e s)^6 \|v\|_2 \quad \text{for all } I \subset [n], \; |I| \ge \e n.
  $$
  Here $C = C(K,M) \ge 1$.
\end{theorem}

Note that we do not require any moments for the matrix entries, so heavy-tailed distributions are allowed.
However, the boundedness assumption formalized by \eqref{eq: BAM}
implicitly yields some upper bound on the tails. Indeed, if the entries of $A$
are i.i.d. and mean zero, then $\|A\| = O(\sqrt{n})$ can only hold if the
fourth moments of entries are bounded.

Further, we do not require that the entries of $A$ have mean zero.
Therefore, adding to $A$ any fixed matrix of norm $O(\sqrt{n})$ does not affect our results.

\medskip

Extending Theorem \ref{thm: delocalization continuous} to general, possibly discrete distributions,
is a challenging task. We are able to do this for matrices with identically distributed entries
and under the mild assumption that the distributions of entries are not too concentrated near a single number.

\begin{assumption}[General distribution of entries]	\label{A: general distribution}
  We assume that the real parts of the matrix entries are
  distributed identically with a random variable $\xi$ that satisfies
  \begin{equation}							\label{eq: xi}
  \sup_{u \in \R} \Pr{|\xi-u| \le 1} \le 1-p, \quad \Pr{|\xi| > K} \le p/2 \quad \text{for some } K, p >0.
  \end{equation}
\end{assumption}

Among many examples of discrete random variables $\xi$ satisfying Assumption~\ref{A: general distribution},
the most prominent one is the symmetric {\em Bernoulli} random variable $\xi$,
which takes values $-1$ and $1$ with probability $1/2$ each.

With Assumption \ref{A: continuous distribution} replaced by Assumption~\ref{A: general distribution}, we can prove the no-gaps delocalization result, which we summarize in the following theorem.

\begin{theorem}[Delocalization: general distributions]	\label{thm: delocalization general}
  Let $A$ be an $n \times n$ random matrix which satisfies Assumptions \ref{A} and \ref{A: general distribution}.
  Let $M \ge 1$.
  Let $\e \ge 1/n$ and $s \ge c_1 \e^{-7/6} n^{-1/6} +e^{-c_2/\sqrt{\e}}$.
  Then,  the following event holds with probability
  at least
  \[
    1-(C s)^{\e n}-\P(\BB_{A,M}^c).
  \]
  Every eigenvector $v$ of $A$ satisfies
  $$
  \|v_I\|_2 \ge (\e s)^6 \|v\|_2 \quad \text{for all } I \subset [n], \; |I| \ge \e n.
  $$
  Here $c_k = c_k(p,K,M)>0$ for $k=1, 2$ and $C = C(p,K,M) \ge 1$.
\end{theorem}

\begin{remark}  \label{rem: shift}
The proof of Theorem \ref{thm: delocalization general} presented in \cite{RV no-gaps} can be modified to allow an extension to random matrices shifted by a constant multiple of the all ones matrix $\mathbf{1}_n$. More precisely, for a given $\mu \in \C$, the event discribed in the theorem holds with probability at least $1-(C s)^{\e n}-\P(\BB_{A- \mu \mathbf{1}_n,M}^c)$. This allows to consider random matrices with Bernoulli$(p)$ entries for $p$ being a constant. Moreover, tracing the proof appearing in \cite{RV no-gaps}, one can see that the constants $c_k$ and $C$ depend polynomially on $p$, which allows to extend no-gaps  delocalization to matrices with i.i.d. Bernoulli entries for $p=\Omega(n^{-c'})$ for some absolute constant $c' \in (0,1)$.
\end{remark}

\begin{remark}
 The no-gaps delocalization phenomenon holds also for any unit vector which is a linear combination of eigenvectors whose eigenvalues are not too far apart, see Remark \ref{rem: approximate eigenvectors}  for the details.
\end{remark}

\subsection*{Acknowledgement}
These notes are based on the mini-courses given  at Hebrew University of Jerusalem and at PCMI Summer School on Random Matrices. The author is grateful to Alex Samorodnitsky, Alexey Borodin, Ivan Corwin, and Alice Guionnet for their hospitality and an opportunity to present this material. The author is also grateful to Feng Wei for for running problem sessions at PCMI which were an integral part of the mini-course and for careful reading of the manuscript.

\section{Reduction of no-gaps delocalization to invertibility of submatrices}

\subsection{From no-gaps delocalization to the smallest singular value bounds}
The first step in proving no-gaps delocalization is pretty straightforward. Let us consider the toy case when there exists a unit eigenvector $u$ of the matrix $A$ with $u_j=0$ for all $j \in J$, where $J$ is some subset of $[n]$. If we denote the corresponding eigenvalue by $\l$ and the submatrix of $A$ with columns from the set $J^c$ by $A_{J^c}$, then we have that
$(A_{J^c} - \l I_{J^c}) u_{J^c}=0$ so the kernel of $A_{J^c} - \l I_{J^c}$ is non-trivial.
Here, $A_{J^c} - \l I_{J^c}$ is a ``tall'' matrix with the number of rows larger than the number of columns. A linear operator defined by a tall rectangular random matrix with sufficiently many independent entries is an injection with high probability. This means that the event that the probability of this ``toy'' case should be small. This idea is not directly applicable since the random eigenvalue $\l$ depends on all entries of the matrix $A$, but this this obstacle is easy to circumvent by discretizing the set of plausible values of $\l$ and considering a deterministic $\l$ from this discretization. If the probability that $A_{J^c} - \l I_{J^c}$ is close to a singular matrix is small for any fixed $\l$, we can use the union bound over the dicretisation along with approximation to show that, with high probability, the matrix $A_{J^c} - \l I_{J^c}$ has a trivial kernel for all $\l$ from this plausible set simultaneously. This would imply the same statement for a random $\l$ allowing us to avoid using hard to obtain information about its distribution except for a very rough bound defining the plausible set.

To implement this idea for a real setup, recall the definition of the singular values of a matrix. Let $B$ be a real or complex $N \times n$ matrix, $N \ge n$. The singular values of $B$ are defined as the square roots of eigenvalues of $B^* B$ arranged in the decreasing order:
\[
 s_1(B) \ge s_2(B) \ge \ldots \ge s_n(B) \ge 0.
\]
If $B$ is real, and we consider this matrix as a linear operator $B:\R^n \to \R^N$, then the image of the Euclidean unit ball will be an ellipsoid whose semi-axes have lengthes $s_1(B) \etc s_n(B)$. The extreme singular values have also an analytic meaning with
\begin{align*}
 s_1(B)& = \max_{x \in S^{n-1}} \norm{Bx}_2 \quad \text{and} \\
 s_n(B)&=\min_{x \in S^{n-1}} \norm{Bx}_2,
\end{align*}
so $s_1(B)=\norm{B}$ -- the operator norm of $B$, and $s_n(B)$ is the distance from $B$ to the set of matrices of a rank smaller than $n$ in the  operator norm.
Throughout these notes, we will also denote the smallest singular value by $s_{\min}(B)$.
We will also abbreviate $A- \l I$ to $A-\l$.

Let us introduce the event that one of the eigenvectors is localized. Define the \emph{localization event} by
\[
\text{Loc}(A, \e, \d) := \left\{ \exists \text{ eigenvector } v \in S_{\C}^{n-1}, \, \exists I \subset [n], \; |I| = \e n:
\|v_I\|_2 < \d \right\}.
\]
Since we assume in Theorem~\ref{thm: delocalization continuous} that the boundedness
event $\BB_{A,M}$ holds with probability at least $1/2$, the conclusion of that theorem can be stated as follows:
\begin{equation}         \label{eq: PLB}
\Pr{ \text{Loc}(A, \e, (\e s)^6) \text{ and } \BB_{A,M} } \le (cs)^{ \e  n}.
\end{equation}
The following proposition reduces proving a delocalization result like \eqref{eq: PLB} to an invertibility bound.

\begin{proposition}[Reduction of delocalization to invertibility]			\label{prop: reduction}
  Let $A$ be an $n \times n$ random matrix with arbitrary distribution.
  Let $M \ge 1$ and $\e, p_0, \d \in (0,1/2)$.
  Assume that for any number $\l_0 \in \C$, $|\l_0| \le M \sqrt{n}$,
  and for any set $I \subset [n]$, $|I| = \e n$, we have
  \begin{equation}         \label{eq: invertibility assumption}
  \Pr{s_{\min} \big( (A - \l_0)_{I^c} \big) \le 8 \d M \sqrt{n} \text{ and } \BB_{A,M} } \le p_0.
  \end{equation}
  Then
  $$
  \Pr{ \text{Loc}(A, \e, \d) \text{ and } \BB_{A,M} } \le 5\d^{-2} (e/\e)^{\e n} p_0.
  $$
\end{proposition}

\begin{proof}
Assume that both the localization event and the boundedness event $\BB_{A,M}$ occur.
Using the definition of $\text{Loc}(A, \e, \d)$, choose a localized eigenvalue-eigenvector pair $(\l, v)$
and an index subset $I$.
Decomposing the eigenvector as
$$
v = v_I + v_{I^c}
$$
and multiplying it by $A-\l$, we obtain
\begin{equation} \label{eq: precise eigenvalue}
0 = (A-\l) v = (A-\l)_I v_I + (A-\l)_{I^c} v_{I^c}.
\end{equation}
By triangle inequality, this yields
$$
\|(A-\l)_{I^c} v_{I^c}\|_2
= \|(A-\l)_I v_I\|_2
\le (\|A\| + |\l|) \|v_I\|_2.
$$
By the localization event $\text{Loc}(A, \e, \d)$, we have $\|v_I\|_2 \le \d$.
By the boundedness event $\BB_{A,M}$ and since $\l$ is an eigenvalue of $A$,
we have $|\l| \le \|A\| \le M \sqrt{n}$. Therefore
\begin{equation}         \label{eq: A-l v}
\|(A-\l)_{I^c} v_{I^c}\|_2 \le 2M \d \sqrt{n}.
\end{equation}

This happens for some $\lambda$ in the disc $\{z \in \C: |z| \le M \sqrt{n}\}$.
We will now run a covering argument in order to fix $\lambda$.
Let $\NN$ be a $(2M \d \sqrt{n})$-net of that disc.
One can construct $\NN$ so that
$$
|\NN| \le \frac{5}{\d^2}.
$$
Choose $\l_0 \in \NN$ so that $|\l_0 - \l| \le 2M \d \sqrt{n}$. By \eqref{eq: A-l v}, we have
\begin{equation}         \label{eq: A-l0 v}
\|(A-\l_0)_{I^c} v_{I^c}\|_2 \le 4M \d \sqrt{n}.
\end{equation}
Since $\|v_I\|_2 \le \d \le 1/2$, we have $\|v_{I^c}\|_2 \ge \|v\|_2 - \|v_I\|_2 \ge 1/2$.
Therefore, \eqref{eq: A-l0 v} implies that
\begin{equation}         \label{eq: smin A-l0}
s_{\min}((A - \l_0)_{I^c}) \le 8M \d \sqrt{n}.
\end{equation}

Summarizing, we have shown that the events $\text{Loc}(A, \e, \d)$ and $\BB_{A,M}$ imply the
existence of a subset $I \subset [n]$, $|I| = \e n$, and a number $\l_0 \in \NN$, such that
\eqref{eq: smin A-l0} holds.
Furthermore, for fixed $I$ and $\l_0$, assumption \eqref{eq: invertibility assumption} states that
\eqref{eq: smin A-l0} together with $\BB_{A,M}$ hold with probability at most $p_0$.
So by the union bound we conclude that
$$
\Pr{ \text{Loc}(A, \e, \d) \text{ and } \BB_{A,M} } \le \binom{n}{\e n} \cdot |\NN| \cdot p_0
\le \Big(\frac{e}{\e} \Big)^{\e n} \cdot \frac{5}{\d^2} \cdot p_0.
$$
This completes the proof of the proposition.
\end{proof}

\begin{remark} \label{rem: approximate eigenvectors}
A simple analysis of the proof of Proposition \ref{prop: reduction} shows that it holds not only for eigenvectors of the matrix $A$, but for its \emph{approximate} eigenvectors as well. Namely, instead of the event $\text{Loc}(A,\e,\d)$ one can consider the following event
\begin{multline*}
   \widetilde{\text{Loc}}(A, \e, \d)
   := \left\{ \exists v \in S_{\C}^{n-1}, \ \exists \l \in \C \ |\l| \le M \sqrt{n} \  \exists I \subset [n], \ |I| = \e n: \right .\\
    \left . \norm{(A - \l I) v}_2 \le M \d \sqrt{n} \text{ and } \norm{v_I}_2 < \d \right\}.
\end{multline*}
This event obeys the same conclusion as $\text{Loc}(A, \e, \d)$:
  $$
  \Pr{ \widetilde{\text{Loc}}(A, \e, \d) \text{ and } \BB_{A,M} } \le 5\d^{-2} (e/\e)^{\e n} p_0.
  $$
Indeed, equation \eqref{eq: precise eigenvalue} is replaced by
\[
w = (A-\l) v = (A-\l)_I v_I + (A-\l)_{I^c} v_{I^c},
\]
where $w$ is a vector of a norm not exceeding $M \d \sqrt{n}$. This in turn results in replacing $2M \d \sqrt{n}$ by $3M \d \sqrt{n}$ in \eqref{eq: A-l v} and $3M \d \sqrt{n}$ by $4M \d \sqrt{n}$ in \eqref{eq: A-l0 v}. This observation shows, in particular,  that the no-gaps delocalization phenomenon holds for any unit vector which is a linear combination of eigenvectors whose eigenvalues are at most $M \d \sqrt{n}$ apart.

\end{remark}

\subsection{The $\e$-net argument.}
We have reduced the proof of the no-gaps  delocalization to establishing quantitative invertibility of a matrix whose  number of rows is larger than  number of columns. This problem has been extensively studied, so before embarking on the real proof, let us check whether we can apply an elementary bound based on the discretization of the sphere. Assume for simplicity that all entries of the matrix $A$ are real and independent, and the entries are centered and of the unit variance.
 We will formulate the result in a bigger generality than we need at this moment.
\begin{lemma}  \label{l: small subspace}
 Let $M>0$ and let $A$ be an $m \times n$ matrix with real independent entries $A_{i,j}$ satisfying
 \[
  \E a_{i,j}=0, \quad  \E a_{i,j}^2=1, \quad \text{and} \quad \E a_{i,j}^4 \le C.
 \]
  Let $E$ be a linear subspace of $\R^n$ of dimension
 \[
  k=\text{dim}(E)< c\frac{m}{\log(2+n/m)}.
 \]
   Then with probability at least $1-\exp(-c'm)-\Pr{\BB_{A,M}^c}$, all vectors $x \in E$ satisfy
 \[
  \norm{Ax}_2 \ge c \sqrt{m}.
 \]
\end{lemma}
 The proof of Lemma \ref{l: small subspace} is based on the $\e$-net argument.
To implement it, we need an elementary lemma.
\begin{lemma} \label{l: eps-net}
 Let $\e\in (0,1]$ and let $V \subset S_{\R}^{k-1}$ be any set. The set $V$ contains an $\e$-net of cardinality at most $(1+2/\e)^k$.
\end{lemma}

\begin{proof}[Proof of Lemma \ref{l: small subspace}]
  Let $\e>0$.
  It is enough to prove the norm bound for all vectors of $V:=E \cap S^{n-1}$. Since the dimension of $E$ is $k$, this set admits an $\e$-net $\NN$ of cardinality $(1+2/\e)^k$. Let $y \in \NN$, and let $z_j=(Ay)_j$ be the $j$-th coordinate of the vector $Ay$.

  The Paley--Zygmund inequality asserts that a random variable $Y \ge 0$ satisfies
   \[
    \Pr{Y>t} \ge \frac{(\E Y-t)^2}{\E Y^2} \quad \text{for any } t\in (0, \E Y).
   \]
   If $Y=z_j^2$, the assumptions on $a_{i,j}$ imply $\E Y=1$ and $\E Y^2 \le C'$.
   Applying the Paley--Zygmund inequality with $t=1/2$, we conclude
   that $\Pr{|z_j| \ge 1/2} \ge c$. Using Chernoff's inequality, we derive that
  \begin{align}  \label{eq: Hoeffding}
    \Pr{\norm{Ay}_2 \le \frac{1}{4} \sqrt{m} }
    &= \Pr {\sum_{j=1}^m |z_j|^2 \le \frac{1}{16} m }  \notag \\
    &\le  \left( |\{j: \ |z_j| \le 1/2 \}| \ge \frac{m}{2} \right)
    \le \exp(-c_2 m).
  \end{align}
  In combination with the union bound, this yields
  \begin{equation} \label{eq: sm-sub net}
    \Pr{\exists y \in \NN  \ \norm{Ay}_2 \le (1/4) \sqrt{m}}
    \le (1+2/\e)^k \exp(-c_2 m).
  \end{equation}
  Let $\Om$ be the event that $\norm{Ay}_2 > (1/4) \sqrt{m}$ for all $y \in \NN$ intersected with $\BB_{A,M}$. Assuming that $\Om$ occurs, we will show that the matrix is invertible on the whole $V$. To this end, take any $x \in V$, and find $y \in \NN$ such that $\norm{x-y}_2 <\e$. Then
  \[
   \norm{Ax}_2 \ge \norm{Ay}_2- \norm{A} \cdot \norm{x-y}_2
   \ge \frac{1}{4} \sqrt{m} -  M \sqrt{n} \cdot \e
   \ge \frac{1}{8} \sqrt{m}
  \]
  if we set
  \[
    \e= \frac{1}{8M} \cdot \sqrt{\frac{m}{n}} \wedge 1.
  \]
  It remains to estimate the probability that $\Om$ does not occur. By \eqref{eq: sm-sub net},
  \[
   \Pr{\Om^c} \le \exp(k \log (1+2/\e)-c_2 m) +\Pr{\BB_{A,M}^c}
   \le  \exp \left(-\frac{c_2}{2} m \right) +\Pr{\BB_{A,M}^c}
  \]
  if we choose
 \[
  k \le c \frac{m}{\log(2+n/m)}.
 \]
\end{proof}

Comparing the bound \eqref{eq: invertibility assumption} required to establish delocalization with the smallest singular value estimate of lemma \ref{l: small subspace}, we see several obstacles preventing the direct use of the $\e$-net argument.

\subsubsection*{Lack of independence}
    As we recall from Assumption~\ref{A}, we are looking for ways to control symmetric and
    non-symmetric matrices simultaneously.
    This forces us to consider random matrices with dependent entries making Chernoff's inequality unapplicable.

\subsubsection*{Small exceptional probability required}
Lemma \ref{l: small subspace}  provides the smallest singular value bound for rectangular matrices whose number of rows is significantly greater than the number of columns. If we are to apply it in combination with Proposition \ref{prop: reduction}, we would have to assume in addition that $\e> 1-\e_0$ for some small $\e_0<1$. Considering smaller values of $\e$ would require a small ball probability bound better than \eqref{eq: Hoeffding} that we used in the proof. We will show that such bound is possible to obtain in the case when the entries have a bounded density. In the general case, however, such bound is unavailable. Indeed, if the entries of the matrix may take the value 0 with a positive probability, then $\P(Ae_1=0)=\exp(-cm)$, which shows that the bound \eqref{eq: Hoeffding} is, in general, optimal. Overcoming this problem for a general distribution would require a delicate stratification of the unit sphere according to the number-theooretic structure of the coordinates of a vector governing the small ball probability bound.

A closer look at Proposition \ref{prop: reduction} demonstrates that the demands for a small ball probability bound are even higher.
We need that the delocalization result, and thus
    the invertibility bound \eqref{eq: A-l0 v}, hold uniformly over all index subsets $I$\
    of size $\e n$.
    Since there are $\binom{n}{\e n} \sim \e^{-\e n}$ such sets, we would need
    the probability  in \eqref{eq: invertibility assumption} to be at most $\e^{\e n}$.
    Such small exceptional probabilities (smaller than $e^{-\e n}$)
    are hard to achieve in  the general case.

\subsubsection*{Complex entries}
    Even if the  original matrix is real, its eigenvalues may be complex. This observation forces us to
     work with complex random matrices. Extending the known invertibility results to complex matrices
    poses two additional challenges. First, in order to preserve the
    matrix-vector multiplication, we replace a complex $n \times m$ random matrix $B = R + iT$
    by the real $2m \times 2n$ random matrix $\left[ \begin{smallmatrix} R & -T \\ T & R \end{smallmatrix} \right]$.
   The real and imaginary parts $R$ and $T$ each appear twice in this matrix,
   which causes extra dependencies of the entries.
   Besides that, we encounter a major problem while trying to apply the $\e$-net argument to prove the smallest singular value bound.
   Indeed, since we have to consider a real $2m \times 2n$ matrix,
   we will have to construct a net in a subset of the real sphere of dimension $2n$. The size of such net is exponential in the dimension. On the other hand, the number of independent rows of $R$ is only $m$, so the small ball probability will be exponential in terms of $m$. If $m<2n$, the union bound would not be applicable.

Each of these obstacles requires a set of rather advanced tools to deal with in general case, i.e. under Assumption \ref{A: general distribution}. Fortunately, under Assumption \ref{A: continuous distribution}, these problems can be addressed in a much easier way allowing a short and rather non-technical proof. For this reason, we are going to concentrate on the continuous  density case below.

\section{Small ball probability for the projections of random vectors}

\subsection{Density of a marginal of a random vector.}
The proof of the no-gaps delocalization theorem requires a result on the distribution of the marginals of a random vector which is of an independent interest. To simplify the presentation, we will consider a vector with independent coordinates having a bounded density. Let $X=(X_1 \etc X_n)$ be independent real valued random variables with densities $f_{X_1} \etc f_{X_n}$ satisfying
\[
  f_{X_j}(t) \le K \quad \text{for all } j \in [n], \ t \in \R.
\]
The independence implies that the density of the vector is the product of the densities of the coordinates, and so, $f_X(x) \le K^n$ for all $x \in \R^n$. Obviously, we can extend the previous observation to the coordinate projections of $X$ showing that $f_{P_J X}(y) \le K^{|J|}$ for any set $J \subset [n]$ and any $y \in \R^J$ with $P_J$ standing for the coordinate projection of $\R^n$ to $\R^J$. It seems plausible that the same property should be shared by the densities of all orthogonal projections to subspaces $E \subset \R^n$ with the dimension of $E$ playing the role of $|J|$. 
Yet, a simple example shows that this statement fails even in dimension 2. Let $X_1,X_2$ be random variables uniformly distributed  on the interval $[-1/2,1/2]$, and consider the projection on the subspace $E \subset \R^2$ spanned by the vector $(1,1)$. Then $Y=P_E X$ is the normalized sum of the coordinates of $X$:
\[
 P_Y= \frac{\sqrt{2}}{2} \left(X_1+X_2\right).
\]
A direct calculation shows that $f_Y(0)=\sqrt{2}>1$. A delicate result of Ball \cite{B1} shows that this is the worst case for the uniform distribution. More precisely, consider a vector $X \in \R^n$ with i.i.d. coordinates uniformly distributed in the interval $[-1/2,1/2]$. Then the projection of $X$ onto any one-dimensional subspace $E=\text{span}(a)$ with $a=(a_1 \etc a_n) \in S^{n-1}$ is a weighted linear combination of the coordinates:
$P_E(X)= \sum_{j=1}^n a_j X_j$. The theorem of Ball asserts that the density of such linear combination does not exceed $\sqrt{2}$ making $a=(\sqrt{2}/2, \sqrt{2}/2, 0 \etc 0)$ the worst sequence of weights. This result can be combined with a theorem of Rogozin claiming that the density of a linear combination of independent random variables increases increases the most if these variables are uniformly distributed. This shows that if the coordinate of $X$ are independent absolutely continuous random variables having densities uniformly bounded by $K$, then the density of $Y= \sum_{j=1}^n a_j X_j$ does not exceed $\sqrt{2}K$ for any $a=(a_1 \etc a_n) \in S^{n-1}$.

Instead of discussing the proofs of the theorems of Ball and Rogozin, we will present here a simpler argument due to Ball and Nazarov \cite{BN} showing that the density of $Y$ is bounded by $CK$ for some unspecified absolute constant $C$. Moreover, we will show that this fact allows a multidimensional extension which we formulate in the following theorem \cite{RV small ball}.
\begin{theorem}[Densities of projections]					\label{thm: dens PX}
  Let $X= (X_1,\ldots, X_n)$ where
  $X_i$ are real-valued independent random variables.
  Assume that the densities of $X_i$ are bounded by $K$
  almost everywhere.
  Let $P$ be the orthogonal projection in $\R^n$ onto a $d$-dimensional subspace.
  Then the density of the random vector $PX$ is bounded by $(CK)^d$ almost everywhere.
\end{theorem}

This theorem shows that the density bound $K^d$ for coordinate projections holds also for general ones if we include a multiplicative factor depending only on the dimension.
Recently, Livshyts et al.~\cite{LPP} proved a multidimensional version of Rogozin's theorem. Combining it with the multidimensional version of Ball's theorem \cite{B2}, one can show that the optimal value of the constant $C$ is $\sqrt{2}$ as in the one-dimensional case.

\begin{proof}
We  will start the proof from the one-dimensional case. The proof in this case is a nice illustration of the power of characteristic functions approach in deriving the small ball and density estimates. As before, we restate the one-dimensional version of the theorem as a statement about the density of a linear combination.

\vskip 0.1in
\textbf{Step 1.}
 Linear combination of independent random variables.

  Let $X_1,\ldots, X_n$ be real-valued independent random variables
  whose densities are bounded by $K$ almost everywhere.
  Let $a_1,\ldots,a_n$ be real numbers with $\sum_{j=1}^n a_j^2 = 1$.
  Then the density of $\sum_{j=1}^n a_j X_j$ is bounded by
  $CK$ almost everywhere.
\vskip 0.05in
We begin with a few easy reductions.
By replacing $X_j$ with $KX_j$ we can assume that $K=1$.
By replacing $X_j$ with $-X_j$ when necessary we can assume that all $a_j \ge 0$.
We can further assume that $a_j > 0$ by dropping all zero terms from the sum.
If there exists $j_{0}$ with $a_{j_0} > 1/2$, then the conclusion follows by conditioning
on all $X_j$ except $X_{j_0}$. Thus we can assume that
$$
0 < a_j < \frac{1}{2} \quad \text{for all } j.
$$
Finally, by translating $X_j$ if necessary we reduce the problem to bounding the
density of $S = \sum_j a_j X_j$ at the origin.

After these reductions, we proceed to bounding $f_S(0)$ in terms of the characteristic function
\[
 \phi_S(t)= \E e^{ i t S}.
\]
We intend to use the Fourier inversion formula
\[
 f_S(0)= \frac{1}{2 \pi} \int_{\R} \phi_S ( x) \, dx.
\]
This formula requires the assumption that $\phi_S \in L_1(\R)$, while we only know that $\norm{\phi_S}_{\infty}\le 1$. This, however, is not a problem. We can add an independent $N(0, \sigma^2)$ random variable to each coordinate of $X$. In terms of the characteristic functions, this amounts to multiplying each $\phi_{X_j}  \in L_{\infty}(\R)$ by a scaled gaussian density making it an $L_1$-function. The bound on the density we are going to obtain will not depend on $\sigma$ which would allow taking $\sigma \to 0$.

By independence of the coordinates of $X$, $\phi_S(x) = \prod_j \phi_{X_j} (a_j t)$.
Combining it with the Fourier inversion formula, we obtain
\begin{equation}         \label{eq: I as product}
f_S(0)
 =  \frac{1}{2 \pi} \int_{\R} \prod_j \phi_{X_j} (a_j x) \, dx
 \le \frac{1}{2 \pi} \prod_j \Big( \int_\R |\phi_{X_j} (a_j x)|^{1/a_j^2} \, dx \Big)^{a_j^2},
\end{equation}
where we used Holder's inequality with exponents $1/a_j^2$ whose reciprocals sum up to 1.

We will estimate each integral appearing in the right hand side of \eqref{eq: I as product} separately. Denote by $\l$ the Lebesgue measure on $\R$. Using the Fubini theorem, we can rewrite each integral as
\begin{equation} \label{eq: distr function formula}
 \frac{1}{a_j} \cdot \int_\R |\phi_{X_j} (x)|^{1/a_j^2} \, dx
 = \int_0^1 \frac{1}{a_j^{3}} \cdot t^{1/a_j^2-1} \l \{x: |\phi_{X_j}(x)| >t \} \, dt.
\end{equation}
To estimate the last integral, we need a bound on the measure of points where the characteristic function is large. Such bound is provided in the lemma below.
\begin{lemma}[Decay of a characteristic function]				\label{l: char function decay}
  Let $X$ be a random variable whose density is bounded by $1$.
  Then the characteristic function of $X$ satisfies
  \[
  \l \{x: |\phi_{X}(x)| >t \} \le
  \begin{cases}
    \frac{2 \pi}{t^2}, & t \in (0,3/4) \\
    C\sqrt{1-t^2}, & t \in[3/4,1].
  \end{cases}
  \]
\end{lemma}

Let us postpone the proof of the lemma for a moment and finish the proof of the one-dimensional case of Theorem \ref{thm: dens PX}.
 Fix $j \in [n]$ and denote for shortness $p=1/a_j^2 \ge 4$. Combining Lemma \ref{l: char function decay} and \eqref{eq: distr function formula}, we obtain
\begin{align*}
 &\frac{1}{a_j} \cdot \int_\R |\phi_{X_j} (x)|^{1/a_j^2} \, dx \\
 &\le p^{3/2} \cdot
   \left( \int_0^{3/4} t^{p-1} \cdot \frac{2 \pi}{t^2} \, dt
   + \int_{3/4}^1 t^{p-1} \cdot C\sqrt{1-t^2} \, dt
   \right) \\
 &\le p^{3/2} \cdot
   \left(  \frac{2 \pi}{p-2}  \cdot (3/4)^{p-2}
   + C \int_0^{\sqrt{7}/4} (1-s^2)^{(p-2)/2} \cdot s^2 \, ds
   \right),
\end{align*}
where we used the substitution $s^2=1-t^2$ in the second term.
The function
\[
 u(p)=p^{3/2} \cdot
    \frac{2 \pi}{p-2}  \cdot (3/4)^{p-2}
\]
is uniformly bounded for $p \in [4, \infty)$. To estimate the second term, we can use the inequality $1-s^2 \le \exp(-s^2)$, which yields
\[
  p^{3/2} \int_0^{\sqrt{7}/4} (1-s^2)^{(p-2)/2} \cdot s^2 \, ds
  \le p^{3/2} \int_0^\infty \exp\left(- \frac{p-2}{2}s^2 \right) s^2 \, ds.
\]
The last expression is also uniformly bounded for $p \in [4, \infty)$.
This proves that
\[
  \frac{1}{a_j} \cdot \int_\R |\phi_{X_j} (x)|^{1/a_j^2} \, dx \le C
\]
for all $j$, where $C$ is an absolute constant.
Substituting this into \eqref{eq: I as product} and using that $\sum_{j=1}^n a_j^2=1$ yields $f_s(0) \le C'$ completing the proof of Step 1 modulo Lemma \ref{l: char function decay}.
\end{proof}
Let us prove the lemma now.
\begin{proof}[Proof of Lemma \ref{l: char function decay}]
 The first bound in the lemma follows from Markov's inequality
 \[
  \l \{x: |\phi_{X}(x)| >t \} \le \frac{\norm{\phi_X}_2^2}{t^2}
 \]
 To estimate the $L_2$-norm, we apply the Plancherel identity:
 \begin{equation}  \label{eq: phi t large}
  \norm{\phi_X}_2^2 =2 \pi \norm{f_X}_2^2
  \le 2 \pi \norm{f_X}_{\infty} \cdot \norm{f_X}_1 \le 2 \pi.
 \end{equation}

 The estimate for $t \in [3/4,1]$ will be based on a regularity argument
going back to Halasz \cite{Halasz 75}.

We will start with the symmetrization.
Let $X'$ denote an independent copy of $X$. Then
\begin{align*}
|\phi_X(t)|^2
&= \E e^{itX} \, \E \overline{e^{itX}}
= \E e^{itX} \, \E e^{-itX'}
= \E e^{it(X-X')} \\
&= \phi_{\tilde{X}}(t),
\quad \text{where } \tilde{X} := X-X'.
\end{align*}
Further, by symmetry of the distribution of $\tilde{X}$, we have
\[
\phi_{\tilde{X}}(t) = \E \cos(t \tilde{X})
= 1 - 2 \E \sin^2 \left( \frac{1}{2} t \tilde{X} \right)
=: 1 - \psi(t).
\]
Denoting $s^2=1-t^2$, we see that to prove that
\[
  \l \{x: \ |\phi_X(x)| >t \} \le C \sqrt{1-t^2} \quad \text{for } t \in [3/4,1],
\]
it is enough to show that
\begin{equation}         \label{eq: mu psi}
\lambda\{ \tau: \psi(\tau) \le s^2 \} \le Cs, \quad \text{for } 0 < s \le 1/2.
\end{equation}

Observe that \eqref{eq: mu psi} holds for some fixed constant value of $s$.
This follows from the identity $|\phi_X(\tau)|^2 = 1 - \psi(\tau)$ and inequality \eqref{eq: phi t large}:
\begin{equation}         \label{eq: fixed t}
\lambda \big\{ \tau: \psi(\tau) \le \frac{1}{4} \big\}
= \lambda \{ \tau: |\phi_X(\tau)| \ge \sqrt{3/4} \}
\le 8\pi/3 \le 9.
\end{equation}
Next, the definition of $\psi(\cdot)$ and
the inequality $|\sin(mx)| \le m|\sin x|$ valid for $x \in \R$ and $m \in \N$ imply
that
$$
\psi(mt) \le m^2 \psi(t), \quad t > 0, \; m \in \N.
$$
Therefore
\begin{equation}         \label{eq: t discrete}
\lambda \big\{ \tau: \psi(\tau) \le \frac{1}{4m^2} \big\}
\le \lambda \big\{ \tau: \psi(m\tau) \le \frac{1}{4} \big\}
= \frac{1}{m} \, \lambda \big\{ \tau: \psi(\tau) \le \frac{1}{4} \big\}
\le \frac{9}{m},
\end{equation}
where in the last step we used \eqref{eq: fixed t}.
This establishes \eqref{eq: mu psi} for the discrete set of values $t = \frac{1}{2m}$, $m \in \N$.
We can extend this to arbitrary $t>0$ in a standard way,
by applying \eqref{eq: t discrete} for $m \in \N$ such that
$t \in (\frac{1}{4m}, \frac{1}{2m}]$.
This proves \eqref{eq: mu psi} and completes the proof of Lemma~\ref{l: char function decay}.
\end{proof}
\vskip 0.1in
We now pass to the multidimensional case. As for one dimension, our strategy will depend on
whether all vectors $Pe_j$ are small or some $Pe_j$ are large.
In the first case, we proceed with a high-dimensional version of the argument
from Step 1, where H\"older's inequality will be replaced by Brascamp-Lieb's inequality.
In the second case, we will remove the large vectors $Pe_j$ one by one, using the induction over the dimension.

\vskip 0.1in
\textbf{Step 2.}
  Let $X$ be a random vector and $P$ be a projection
  which satisfy the assumptions of Theorem~\ref{thm: dens PX}.
  Assume that
  $$
  \|Pe_j\|_2 \le 1/2 \quad \text{for all } j=1,\ldots,n.
  $$
  Then the density of the random vector $PX$ is bounded by $(CK)^d$ almost everywhere.
\vskip 0.1in

The proof will be based on Brascamp-Lieb's inequality.

\begin{theorem}[Brascamp-Lieb \cite{BL}, see also \cite{B2}]
  Let $u_1,\ldots,u_n \in \R^d$ be unit vectors and $c_1,\ldots,c_n >0$ be real numbers satisfying
  $$
  \sum_{i=1}^n c_j u_j u_j^\top = I_d.
  $$
  Let $f_1,\ldots,f_n : \R \to [0,\infty)$ be integrable functions. Then
  $$
  \int_{\R^n} \prod_{j=1}^n f_j(\pr{x}{u_j} )^{c_j} \; dx
  \le \prod_{j=1}^n \Big( \int_\R f_j(t) \; dt \Big)^{c_j}.
  $$
\end{theorem}
A short and very elegant proof of the Brascamp-Lieb inequality based on the measure transportation ideas can be found in \cite{Barthe}.

The singular value decomposition of $P$ yields the existence of a
$d \times n$ matrix $R$ satisfying
$$
P = R^\top R, \quad R R^\top = I_d.
$$
It follows that $\|Px\|_2 = \|Rx\|_2$ for all $x \in \R^d$.
This allows us to work with the matrix $R$ instead of $P$.
As before,  replacing each $X_j$ by $KX_j$, we may assume that $K=1$. Finally, translating $X$ if necessary we reduce the problem to bounding
the density of $RX$ at the origin.

As in the previous step, Fourier inversion formula
associated with the Fourier transform in $n$ dimensions yields that the density of $RX$ at the origin
can be reconstructed from its Fourier transform as
\begin{equation}         \label{eq: L1 char function}
f_{RX}(0)
= (2 \pi)^{-d}  \int_{\R^d} \phi_{RX}( x) \; dx
\le (2 \pi)^{-d} \int_{\R^d} |\phi_{RX}(x)| \; dx
\end{equation}
where
\begin{equation}         \label{eq: phiRX}
\phi_{RX}(x) = \E \exp \big( i \pr{x}{RX} \big)
\end{equation}
is the characteristic function of $RX$.
Therefore, to complete the proof, it suffices to bound the integral in the
right hand side of \eqref{eq: L1 char function} by $C^d$.

In order to represent $\phi_{RX}(x)$
more conveniently for application of Brascamp-Lieb inequality,
we denote
$$
a_j := \|Re_j\|_2, \quad u_j := \frac{Re_j}{\|Re_j\|_2}.
$$
Then $R = \sum_{j=1}^n a_j u_j e_j^\top$, so the identity $RR^\top= I_d$ can be written as
\begin{equation}         \label{eq: uj decomposition}
\sum_{j=1}^n a_j^2 u_j u_j^\top = I_d.
\end{equation}
Moreover, we have $\pr{x}{RX} = \sum_{i=1}^n a_j \pr{x}{u_j} X_j$.
Substituting this into \eqref{eq: phiRX} and using independence, we obtain
$$
\phi_{RX}(x) = \prod_{j=1}^n \E \exp \big( i a_j \pr{x}{u_j} X_j \big).
$$
Define the functions $f_1, \ldots, f_n : \R \to [0,\infty)$ as
$$
f_j(t) := \big| \E \exp(i a_j t X_j) \Big|^{1/a_j^2} = \big| \phi_{X_j} (a_j t) \big|^{1/a_j^2}.
$$
Recalling \eqref{eq: uj decomposition}, we apply Brascamp-Lieb inequality
for these functions and obtain
\begin{align}			\label{eq: int by product}
\int_{\R^d} |\phi_{RX}(x)| \; dx
  &= \int_{\R^d} \prod_{j=1}^n f_j \big( \pr{x}{u_j} \big)^{a_j^2} \; dx  \nonumber\\
  &\le \prod_{j=1}^n \Big( \int_\R f_j(t) \; dt \Big)^{a_j^2}
  = \prod_{j=1}^n \Big( \int_\R \big| \phi_{X_j} (a_j t) \big|^{1/a_j^2} \; dt \Big)^{a_j^2}.
\end{align}
We arrived at the same quantity as we encountered in one-dimensional argument in \eqref{eq: I as product}.
Following that argument, which uses the assumption that all $a_j \le 1/2$, we bound the product
 above by
$$
(2C)^{\sum_{j=1}^n a_j^2}.
$$
Recalling that $a_j = \|Re_j\|_2$ and , we find that
\[
  \sum_{j=1}^n a_j^2 = \sum_{j=1}^n \|Re_j\|_2^2 = \text{Tr}(RR^\top) = \text{Tr}(I_d) = d.
\]
Thus the right hand side of \eqref{eq: int by product} is bounded by $(2C)^d$.
The proof of Theorem \ref{thm: dens PX} in the case where all $\norm{P e_j}_2$ are small is complete.

\vskip 0.1in
\textbf{Step 3.}
 Inductive argument.
\vskip 0.05in

We will prove Theorem \ref{thm: dens PX} by induction on the rank of the projection.
 The case $\text{rank}(P)=1$ has been already established. We have also proved the Theorem when $\norm{Pe_j}_2<1/2$ for all $j$. Assume that the theorem holds for all projections $Q$ with $\text{rank}(Q)=d-1$ and $\norm{Pe_1}_2 \ge 1/2$.

The density function is not a convenient tool to run the inductive argument since the density of $P_X$ does not usually splits into a product of densities related to the individual coordinates.
Let us consider the \emph{L\'evy concentration function} of a random vector which would replace the density in our argument. 
\begin{definition}
Let $r>0$. 
For a random vector $Y \in \R^n$, define its L\'evy concentration function by
\[
 \LL(Y,r):=\sup_{y \in \R^n} \Pr{\norm{Y-y}_2 \le r}.
\]
\end{definition}
Note that the condition that the density of $Y$ is bounded  is equivalent to
\[
 \LL(Y,r \sqrt{n})\le (Cr)^n \quad \text{for any } r>0.
\]
This follows from  the Lebesgue differentiation theorem and the fact that the Lebesgue measure of a ball of of radius $r \sqrt{n}$ is $(cr)^n$.

In terms of the L\'evy concentration function, the statement of the theorem is equivalent to the claim that for for any $y \in P \R^n$ and any $t > 0$,
\begin{equation} \label{eq: concentration function in d}
 \Pr{\norm{PX-y}_2 \le t \sqrt{d}} \le (M t)^d
\end{equation}
for some absolute constant $M$, where we denoted $d= \text{rank}(P)$. One direction of this equivalence follows from the integration of  the density function over the ball of radius $t \sqrt{d}$ centered at $y$; another one from the Lebesgue differentiation theorem. The induction assumption then reads
\begin{equation} \label{eq: concentration function in d-1}
 \Pr{\norm{QX-z}_2 \le t \sqrt{d-1}} \le (M t)^{d-1}
\end{equation}
for all projections $Q$ of rank $d-1$, $z \in Q \R^n$, and $t > 0$. Comparison of \eqref{eq: concentration function in d-1} and \eqref{eq: concentration function in d} immediately shows the difficulties we are facing: the change from $d-1$ to $d$ in the left hand side of these inequalities indicates that we have to work accurately to preserve the constant $M$ while deriving \eqref{eq: concentration function in d} from \eqref{eq: concentration function in d-1}. This is achieved by a delicate tensorization argument. By considerind an appropriate shift of $X$, we can assume without loss of generality that $y=0$. Let us formulate the induction step as a separate proposition.
\begin{proposition}[Removal of large $Pe_i$]				\label{prop: Pei large}			
  Let $X$ be a random vector satisfying the assumptions of Theorem~\ref{thm: dens PX} with $K=1$,
  and let $P$ be an orthogonal projection in $\R^n$ onto a $d$-dimensional subspace.
  Aassume that
  $$
  \|Pe_1\|_2 \ge 1/2.
  $$
  Define $Q$ to be the orthogonal projection in $\R^n$ such that
  $$
  \ker(Q) = \Span\{ \ker(P), P e_1 \}.
  $$
  Let $M \ge C_0$ where $C_0$ is an absolute constant. If
  \begin{equation}         \label{eq: QX hypothesis}
  \Pr{ \|QX\|_2 \le t \sqrt{d-1} } \le (M t)^{d-1} \quad \text{for all } t \ge 0,
  \end{equation}
  then
  $$
  \Pr{ \|PX\|_2 \le t \sqrt{d} } \le (M t)^d \quad \text{for all } t \ge 0.
  $$
\end{proposition}

\begin{proof}
Let us record a few basic properties of $Q$.
A straightforward check shows that
\begin{equation}							\label{eq: P-Q}
\text{$P-Q$ is the orthogonal projection onto $\Span(Pe_1)$.}
\end{equation}
It follows that $(P-Q)e_1 = Pe_1$, since the orthogonal projection of $e_1$
onto $\Span(Pe_1)$ equals $Pe_1$. Canceling $Pe_1$ on both sides, we have
\begin{equation}							\label{eq: Qe1}
Qe_1 = 0.
\end{equation}
It follows from \eqref{eq: P-Q} that $P$ has the form
\begin{equation}         \label{eq: Px}
Px = \Big( \sum_{j=1}^n a_j x_j \Big) Pe_1 + Q x \quad \text{for } x = (x_1,\ldots,x_n) \in \R^n,
\end{equation}
where $a_j$ are fixed numbers (independent of $x$).
Substituting $x=e_1$, we obtain using \eqref{eq: Qe1} that $Pe_1 = a_1 Pe_1 + Qe_1 = a_1 Pe_1$.
Thus
\begin{equation}         \label{eq: a1}
a_1 = 1.
\end{equation}
Furthermore, we note that
\begin{equation}         \label{eq: Qx x1}
\text{$Qx$ does not depend on $x_1$}
\end{equation}
since $Qx = Q(\sum_{i=1}^n x_j e_j) = \sum_{i=1}^n x_j Qe_j$ and
$Qe_1 = 0$ by \eqref{eq: Qe1}.
Finally, since $Pe_1$ is orthogonal to the image of $Q$,
the two vectors in the right side of \eqref{eq: Px}
are orthogonal. Thus
\begin{equation}         \label{eq: Px norm}
\|Px\|_2^2 = \Big( \sum_{j=1}^n a_j x_j \Big)^2 \|Pe_1\|_2^2 + \|Qx\|_2^2.
\end{equation}

Now let us estimate $\|PX\|_2$ for a random vector $X$.
We express $\|PX\|_2^2$ using \eqref{eq: Px norm} and \eqref{eq: a1} as
$$
\|P X\|_2^2
= \Big( X_1 + \sum_{j=2}^n a_j X_j \Big)^2 \|Pe_1\|_2^2 + \|QX\|_2^2 \\
=: Z_1^2 + Z_2^2.
$$
Since by \eqref{eq: Qx x1} $Z_2$ is determined by $X_2,\ldots,X_n$ (and is independent of $X_1$),
and $\|Pe_i\|_2 \ge 1/2$ by a hypothesis of the lemma, we have
\begin{align*}
\Pr{Z_1 \le t \;|\; Z_2}
  &\le \max_{X_2,\ldots,X_n} \Pr{ \Big| X_1 + \sum_{j=2}^n a_j X_j \Big| \le t/\norm{P e_1}_2 \;\Big|\; X_2,\ldots,X_n} \\
  &\le \max_{u \in \R} \Pr{|X_1-u| \le 2t}
  \le 2 t.
\end{align*}
The proof of the inductive step thus reduces to a two-dimensional statement, which we formulate as a separate lemma.
\begin{lemma}[Tensorization]					\label{l: tensorization}
  Let $Z_1, Z_2 \ge 0$ be random variables and $K_1, K_2 \ge 0$, $d > 1$ be real numbers.
  Assume that
  \begin{enumerate}
    \item $\Pr{Z_1 \le t \;|\; Z_2} \le 2 t$ almost surely in $Z_2$ for all $t \ge 0$;
    \item $\Pr{Z_2 \le t \sqrt{d-1}} \le (M t)^{d-1}$ for all $t \ge 0$.
  \end{enumerate}
  for a sufficiently large absolute constant $M$.
  Then
  $$
  \Pr{\sqrt{Z_1^2 + Z_2^2} \le t \sqrt{d}} \le (M t)^d \quad \text{for all } t \ge 0.
  $$
\end{lemma}
The proof of the tensorization lemma requires an accurate though straightforward calculation. We write
\[
 \Pr{\sqrt{Z_1^2 + Z_2^2} \le t \sqrt{d}}
 =\int_0^{ t^2 d} \Pr{Z_1 \le (t^2 d-x)^{1/2} \;|\; Z_2^2 = x} \; d F_2(x)
\]
where $F_2(x) = \Pr{Z_2^2 \le x}$ is the cumulative distribution function of $Z_2^2$.
Using hypothesis (1) of the lemma, we can bound the right hand side of  by
$$
2 \int_0^{t^2 d} (t^2 d-x)^{1/2} \; d F_2(x) =  \int_0^{t^2 d} F_2(x) (t^2 d-x)^{-1/2} \; dx,
$$
where the last equation follows by integration by parts.
Hypothesis (2) of the lemma states that
 \[
  F_2(x) \le M^{d-1} \left(\frac{x}{d-1} \right)^{(d-1)/2}.
 \]
Substituting this into the equality above and estimating the resulting integral explicitly, we obtain
\begin{align*}
  &\Pr{\sqrt{Z_1^2 + Z_2^2} \le t \sqrt{d}}
  \le \int_0^{t^2 d}  M^{d-1} \left(\frac{x}{d-1} \right)^{(d-1)/2} (t^2 d-x)^{-1/2} \; dx \\
  &= t^d \cdot M^{d-1} \frac{d^{d/2}}{(d-1)^{(d-1)/2}} \int_0^1 y^{(d-1)/2} (1-y)^{-1/2} \, dy
  \le  t^d \cdot M^{d-1} \cdot C,
\end{align*}
where the last inequality follows with an absolute constant $C$ from the known asymptotic of the beta-function. Alternatively, notice that
\[
  \frac{d^{d/2}}{(d-1)^{(d-1)/2}} \le \sqrt{e d},
\]
and
\begin{align*}
  \int_0^1 y^{(d-1)/2} (1-y)^{-1/2} \, dy
  &\le \int_0^{1-1/d} y^{(d-1)/2} \sqrt{d} \, dy + \int_{1-1/d}^1  (1-y)^{-1/2} \, dy \\
  &\le \frac{2}{\sqrt{e d}} + \frac{1}{2 \sqrt{d}}.
\end{align*}
This completes the proof of the lemma if we assume that $M \ge C$.
\end{proof}

\subsection{Small ball probability for the image of a vector.}

Let us derive an application of Theorem \ref{thm: dens PX} which will be important for us in the proof of the no-gaps delocalization theorem. We will prove a small ball probability estimate for the image of a fixed vector under the action of a random matrix with independent entries of bounded density.
\begin{lemma}[Lower bound for a fixed vector]		\label{lem: fixed vector}
  Let $G$ be an $l \times m$ matrix with independent complex random entries. Assume that the real parts of the entries have uniformly bounded densities, and the imaginary parts are fixed.
  For each $x \in S_{\C}^{m-1}$ and $\theta >0$, we have
  $$
  \Pr{ \|Gx\|_2 \le \theta \sqrt{l} } \le (C_0  \theta)^{l}.
  $$
\end{lemma}

To prove this lemma, let us derive the small ball probability bound for a fixed coordinate of $Gx$ first.

\begin{lemma}[Lower bound for a fixed row and vector]			\label{lem: fixed row and vector}
  Let $G_j$ denote the $j$-th row of $G$.
  Then for each $j$, $z \in S_{\C}^{n-1}$, and $\theta \ge 0$, we have
  \begin{equation}							\label{eq: fixed row vector}
  \Pr{ |\pr{G_j}{z}| \le \theta} \le C_0 K \theta.		
  \end{equation}
\end{lemma}

\begin{proof}
Fix $j$ and consider the random vector $Z = G_j$. Expressing $Z$ and $z$
in terms of their real and imaginary parts as
$$
Z = X+iY, \quad z = x + iy,
$$
we can write the inner product as
$$
\pr{Z}{z} = \left[ \pr{X}{x} - \pr{Y}{y} \right] + i \left[ \pr{X}{y} + \pr{Y}{x} \right].
$$

Since $z$ is a unit vector, either $x$ or $y$ has norm at least $1/2$.
Assume without loss of generality that $\|x\|_2 \ge 1/2$. Dropping the imaginary part, we obtain
$$
|\pr{Z}{z}| \ge \left| \pr{X}{x} - \pr{Y}{y} \right|.
$$
The imaginary part $Y$ is fixed. Thus
\begin{equation}         \label{eq: SBP Zz}
\Pr{ |\pr{Z}{z}| \le \theta} \le \LL( \pr{X}{x}, \theta ).
\end{equation}

We can express $\pr{X}{x}$ in terms of the coordinates of $X$ and $x$ as the sum
$$
\pr{X}{x} = \sum_{k=1}^n X_k x_k.
$$
Here
$X_k$ are independent random variables with densities bounded by $K$.
Recalling that $\sum_{k=1}^m x_k^2 \ge 1/2$,
we can apply Theorem \ref{thm: dens PX} for a rank one projection.
It yields
\begin{equation}         \label{eq: SBP Xx}
\LL( \pr{X}{x}, \theta ) \le C K \theta.
\end{equation}
Substituting this into \eqref{eq: SBP Zz} completes the proof of Lemma~\ref{lem: fixed row and vector}.
\end{proof}

Now we can complete the proof of Lemma \ref{lem: fixed vector}
We can represent $\|Gx\|_2^2$ as a sum of independent non-negative random variables
$\sum_{j=1}^{l} |\pr{G_j}{x} |^2$.
Each of the terms $\pr{G_j}{x}$ satisfies \eqref{eq: fixed row vector}.
Then the conclusion follows from the following Tensorization Lemma applied to $V_j=|\pr{G_j}{x} |$.
\begin{lemma}  \label{lem: tensorisation}
 Let $V_1 \etc V_l$ be independent non-negative random variables satisfying
 \[
  \Pr{V_j<t} \le Ct
 \]
 for any $t>0$. Then
 \[
  \Pr{ \sum_{j=1}^l V_j^2 < t^2 l } \le (ct)^l.
 \]
\end{lemma}

\begin{proof}
 Since the random variables $V_1^2 \etc V_l^2$ are independent as well, the Laplace transform becomes a method of choice in handling this probability. By Markov's inequality, we have
 \begin{align*}
  \Pr{ \sum_{j=1}^l V_j^2 < t^2 l }
 & = \Pr{l- \frac{1}{t^2}  \sum_{j=1}^l V_j^2 >0 }
 \le \E \exp \left( l- \frac{1}{t^2}  \sum_{j=1}^l V_j^2 \right) \\
 & = e^l \prod_{j=1}^l \E \exp(-V_j^2/t^2).
 \end{align*}
 To bound the expectations in the right hand side, we use the Fubini theorem:
 \[
  \E \exp(-V_j^2/t^2)= \int_0^{\infty} 2x e^{-x^2} \Pr{V_j< tx} \, dx \le Ct,
 \]
 where the last inequality follows from the assumption on the small ball probability of $V_j$.
 Combining the previous two inequalities completes the proof.
\end{proof}

\section{No-gaps delocalization for matrices with absolutely continuous entries.}
In this section, we prove Theorem \ref{thm: delocalization continuous}.
To this end, we combine all the tools we discussed above:  the bound on the density of a projection of a random vector obtained in Theorem \ref{thm: dens PX},  the $\e$-net argument, and the small ball probability bound of Lemma \ref{lem: fixed vector}.
\subsection{Decomposition of the matrix}
Let us recall that we have reduced the claim of delocallization Theorem \ref{thm: delocalization continuous} to the following quantitative invertibility problem:
\begin{itemize}
  \item Let $A$ be an $n \times n$ matrix satisfying Assumptions \ref{A} and \ref{A: continuous distribution}. Let $\e>0 , \ t>0$,  $M>1$, and let $\l \in \C, \ |\l| \le M\sqrt{n}$. Let $I \subset[n]$ be a fixed set of cardinality $|I|=\e n$. Estimate
      \[
       p_0:=\P(s_{\min}((A-\l)_{I^c}) < t \sqrt{n} \text{ and } \norm{A} \le M \sqrt{n}).
      \]
\end{itemize}
Since the set $I$ is fixed, we can assume without loss of generality that $I$ consists of the last $\e n$ coordinates.

Let us decompose $(A-\l)_{I^c}$ as follows:
\begin{equation}         \label{eq: A decomposed}
 (A-\l)_{I^c} =
\begin{bmatrix}
\phantom{X} B \phantom{X} \\ \phantom{X} G \phantom{X}
\end{bmatrix},
\end{equation}
where $B$ and $G$ are rectangular matrices of size $(1-\e/2)n \times (1-\e)n$ and $(\e/2) n \times (1-\e)n$ respectively.
By Assumption \ref{A}, the random matrices $B$ and $G$ are independent, and moreover
all entries of $G$ are independent. At the same time, the matrix $B$ is still rectangular, and the ratio of its number of rows and columns is similar to that of the matrix $(A-\l)_{I^c}$. This would allow us to prove a weaker statement for the matrix $B$. Namely, instead of bounding the smallest singular value, which is the minimum of $\norm{Bx}_2$ over all unit vectors $x$, we will obtain the desired lower bound for all vectors which are far away from a certain low-dimensional subspace depending on $B$. The independence of $B$ and $G$ would make it possible to condition on $B$ fixing this subspace and apply Lemma \ref{l: small subspace} to the matrix $G$ restricted to this subspace to ensure that the matrix $(A-\l)_{I^c}$ is well invertible on this space as well.

Following this road map, we are going to show that either $\|Bx\|_2$ or $\|Gx\|_2$ is nicely bounded below for every vector $x \in S_{\C}^{n-1}$.
To control $B$, we use the second negative moment identity
to bound the Hilbert-Schmidt norm of the pseudo-inverse of $B$.
We deduce from it that most singular values of $B$ are not too small --
namely, all but $0.01 \e n$ singular values are bounded below by $\Omega(\sqrt{\e n})$.
It follows that $\|Bx\|_2$ is nicely bounded below when $x$ restricted to a subspace
of codimension $0.01 \e n$. (This subspace is formed by the corresponding singular vectors.)
Next, we condition on $B$ and we use $G$ to control the remaining $0.01 \e n$ dimensions.
Therefore, either $\|Bx\|_2$ or $\|Gx\|_2$ is nicely bounded below on the entire space, and thus
$\| (A-\l)_{I^c} x\|_2$ is nicely bounded below on the entire space as well.

We will now pass to the implementation of this plan.
To simplify the notation, \\ \emph{assume that the maximal density of the entries is bounded by 1}. \\
The general case can be reduced to this by scaling the entries.

\subsection{The negative second moment identity}
Let $k \ge m$. Recall that the Hilbert-Schmidt norm of a $k \times m$ matrix $V$ is just the Euclidean norm of the $km$-dimensional vector consisting of its entries. Like the operator norm, the Hilbert-Schmidt norm is invariant under unitary or orthogonal transformations of the matrix $V$. This allows to rewrite it in two ways:
\[
  \norm{V}_{HS}^2= \sum_{j=1}^m \norm{V_j}_2^2 = \sum_{j=1}^m s_j(V_j)^2,
\]
where $V_1 \etc V_m$ are the columns of $V$, and $s_1(V) \ge s_2(V) \ge \ldots \ge s_m(V) \ge 0$ are its singular values.
Applying this observation to the inverse of the linear operator defined by $V$ considered as an operator from $V \C^m$ to $\C^m$, we obtain \emph{the negative second moment identity}, see \cite{Tao book}:
\[
\sum_{j=1}^{m} s_j(B)^{-2} = \sum_{i=1}^m \dist(B_j, H_j)^{-2}.
\]
Here  $B_j$ denote
the columns of $B$, and $H_j = \Span(B_l)_{l \ne j}$.

Returning to the matrix $B$, denote for shortness $m=(1-\e)n$ and $\e'=\frac{\e}{2(1-\e)}$. In this notation, $B$ is a $(1+\e')m \times m$ matrix.
To bound the sum above, we have to establish a lower bound on the distance
between the random vector $B_j \in \C^{(1+\e')m}$
and random subspace $H_j \subseteq \C^{(1+\e')m}$ of complex dimension $m-1$.

\subsubsection{Enforcing independence of vectors and subspaces}

Let us fix $j$. If all entries of $B$ are independent, then $B_j$ and $H_j$ are independent.
However, Assumption \ref{A} leaves a possibility for $B_j$ to be correlated with $j$-th
row of $B$. This means that $B_j$ and $H_j$ may be dependent, which would
complicate the distance computation.

There is a simple way to remove the dependence by projecting out the $j$-th coordinate.
Namely, let $B'_j \in \C^{(1+\e')m-1}$ denote the vector $B_j$ with $j$-th coordinate removed,
and let $H'_j = \Span(B'_k)_{k \ne j}$. We note the two key facts. First,
$B'_j$ and $H'_j$ are independent by Assumption \ref{A}.
Second,
\begin{equation}         \label{eq: BjHj}
\dist(B_j, H_j) \ge \dist(B'_j, H'_j),
\end{equation}
since the distance between
two vectors can only decrease after removing a coordinate.

Summarizing, we have
\begin{equation}         		\label{eq: second negative moment}
\sum_{j=1}^m s_j(B)^{-2} \le \sum_{j=1}^m \dist(B'_j, H'_j)^{-2}.
\end{equation}

We are looking for a lower bound for the distances $\dist(B'_j, H'_j)$.
It is convenient to represent them via the orthogonal projection of $B'_j$ onto $(H'_j)^\perp$:
\begin{equation}         \label{eq: dist as proj}
\dist(B'_j, H'_j) = \| P_{E_j} B'_j \|_2,
\quad \text{where} \quad
E_j = (H'_j)^\perp.
\end{equation}
Recall that $B'_j \in \C^{(1+\e')m-1}$ is a random vector with independent
entries whose real parts have densities bounded by $1$
(by Assumptions~\ref{A} and \ref{A: continuous distribution});
and $H'_j$ is an independent subspace of $\C^{(1+\e')m-1}$ of complex dimension $m-1$.
This puts us on a familiar ground as we have already proved Theorem \ref{thm: dens PX}. Now, the main strength of this result becomes clear. The bound of Theorem \ref{thm: dens PX} is uniform over the possible subspaces $E_j$ meaning that we do not need any information about the specific position of this subspace in $\C^{(1+\e')m-1}$. This is a major source of simplifications in the proof of Theorem \ref{thm: delocalization continuous} compare to Theorem \ref{thm: delocalization general}. Under Assumption \ref{A: general distribution}, a bound on the small ball probability for $\| P_{E_j} B'_j \|_2$ depends on the arithmetic structure of the vectors contained in the space $E_j$. Identifying subspaces of $\C^{(1+\e')m-1}$ containing vectors having exceptional arithmetic structure and showing that, with high probability, the space $E_j$ avoids such positions, takes a lot of effort. Fortunately, under Assumption \ref{A: continuous distribution}, this problem does not arise thanks to the uniformity mentioned above.

\subsubsection{Transferring the problem from $\C$ to $\R$}				\label{s: from C to R}
If the real and the imaginary part of each entry of $A$ are random variables of bounded density, one can apply Theorem \ref{thm: dens PX} directly.  However, this case does not cover many matrices satisfying Assumption \ref{A}, most importantly, the matrices with real entries and complex spectrum. The general case, when
only the real parts of the vector $B'_j \in \C^{(1+\e')m-1}$ are random, requires an additional symmetrization step. Indeed, if we transfer the problem from the complex vector space to a real one of the double dimension,  only a half of the coordinates will be random.
Such vector would not be absolutely continuous, so we cannot operate in terms of the densities. As in the previous section, the \emph{L\'evy concentration function} of a random vector would replace the density in our argument.

Let us  formally transfer the problem from the complex to the real field.
To this end, we define the operation $z \mapsto \text{Real}(z)$ that makes complex vectors real
in the obvious way:
$$
\text{for } z = x+iy \in \C^N, \text{ define } \text{Real}(z) = \begin{bmatrix}x \\ y \end{bmatrix} \in \R^{2N}.
$$
Similarly, we can make a complex subspace $E \subset \C^N$ real by defining
$$
\text{Real}(E) = \{ \text{Real}(z) :\; z \in E \} \subset \R^{2N}.
$$
Note that this operation doubles the dimension of $E$.

Record two properties that follow straight from this definition.

\begin{lemma}			\label{lem: real}
(Elementary properties of operation $x \mapsto \real(x)$)	
  \begin{enumerate}
    \item \label{part: PE real}
      For a complex subspace $E$ and a vector $z$, one has
      $$
      \real(P_E z) = P_{\real(E)} \real(z).
      $$
    \item \label{part: LL real}
      For a complex-valued random vector $X$ and $r \ge 0$, one has
      $$
      \LL(\real(X),r) = \LL(X,r).
      $$
  \end{enumerate}
\end{lemma}

The symmetrization lemma allows randomizing all coordinates.

\begin{lemma}[Randomizing all coordinates]					\label{lem: Z real}
  Consider a random vector $Z=X+iY \in \C^N$ whose imaginary part $Y \in \R^N$ is fixed.
  Set $\widehat{Z} = \begin{bmatrix}X_1 \\ X_2 \end{bmatrix} \in \R^{2N}$ where $X_1$ and $X_2$ are independent copies of $X$.
  Let $E$ be a subspace of $\C^N$. Then
  $$
  \LL(P_E Z, r) \le \left( \LL(P_{\real(E)} \widehat{Z}, 2r) \right)^{1/2}, \quad r \ge 0.
  $$
\end{lemma}

\begin{proof}
Recalling the definition of the concentration function, in order to bound $\LL(P_E Z, r)$
we need to choose arbitrary $a \in \C^N$ and find a uniform bound on the probability
$$
p := \Pr{ \|P_E Z - a\|_2 \le r }.
$$
By assumption, the random vector $Z = X + iY$ has fixed imaginary part $Y$.
So it is convenient to express the probability as
$$
p = \Pr{ \|P_E X - b\|_2 \le r }
$$
where $b = a - P_E(iY)$ is fixed. Let us rewrite this identity
using independent copies $X_1$ and $X_2$ of $X$ as follows:
$$
p = \Pr{ \|P_E X_1 - b\|_2 \le r } = \Pr{ \|P_E (iX_2) - ib\|_2 \le r }.
$$
(The last equality follows trivially by multiplying by $i$ inside the norm.)
Using the independence of $X_1$ and $X_2$ and the triangle inequality, we obtain
\begin{align*}
p^2
  &= \Pr{ \|P_E X_1 - b\|_2 \le r \text{ and } \|P_E (iX_2) - ib\|_2 \le r } \\
  &\le \Pr{ \|P_E (X_1 + iX_2) - b - ib \|_2 \le 2r } \\
  &\le \LL(P_E (X_1 + iX_2), 2r).
\end{align*}
Further, using part~\ref{part: LL real} and then part~\ref{part: PE real} of Lemma~\ref{lem: real}, we
see that
\begin{align*}
\LL(P_E (X_1 + iX_2), 2r)
 &= \LL(P_{\real(E)} (\real(X_1 + iX_2)), 2r) \\
 & = \LL(P_{\real(E)} \widehat{Z}, 2r).
\end{align*}
Thus we showed that $p^2 \le \LL(P_{\real(E)} \widehat{Z}, 2r)$ uniformly in $a$.
By definition of the L\'evy concentration function, this completes the proof.
\end{proof}

\subsubsection{Bounding the distances below}

We are ready to control the distances appearing in \eqref{eq: dist as proj}.

\begin{lemma}[Distance between random vectors and subspaces]		\label{lem: distance}
  For every $j \in [n]$ and $\tau>0$, we have
  \begin{equation}         \label{eq: distance}
  \Pr{\dist(B'_j, H'_j) < \tau \sqrt{\e' m}} \le (C  \tau)^{\e' m}.
  \end{equation}
\end{lemma}

\begin{proof}
Representing the distances via projections of $B_j'$ onto the subspaces $E_j = (H'_j)^\perp$
as in \eqref{eq: dist as proj}, and using the definition of the L\'evy concentration function, we have
$$
p_j := \Pr{\dist(B'_j, H'_j) < \tau \sqrt{\e' m}}
\le \LL(P_{E_j} B'_j, \, \tau \sqrt{\e' m}).
$$
Recall that $B_j'$ and $E_j$ are independent, and let us condition on $E_j$.
Lemma~\ref{lem: Z real} implies that
$$
p_j \le \left(\LL(P_{\real(E_j)} \widehat{Z}, \, 2 \tau \sqrt{\e' m}) \right)^{1/2}
$$
where $\widehat{Z}$ is a random vector with independent coordinates
that have densities bounded by 1.

The space $H'_j$ has codimension $\e' m$; thus $E_j$ has dimension $\e' m$
and $\real(E_j)$ has dimension $2\e' m$.
By Theorem \ref{thm: dens PX}, the density of $P_{\real(E_j)} \widehat{Z}$ is bounded by $C^{2\e' m}$.
Integrating the density over a ball of radius $2 \tau \sqrt{\e' m}$ in the subspace $\real(E_j)$
which has volume $(C\tau)^{2\e' m}$, we conclude that
$$
\LL(P_{\real(E_j)} \widehat{Z}, \,  2 \tau \sqrt{\e n}) \le (C  \tau)^{2 \e' m}.
$$
It follows that
$$
p_j \le (C  \tau)^{\e' m},
$$
as claimed. The proof of Lemma~\ref{lem: distance} is complete.
\end{proof}

\subsection{$B$ is bounded below on a large subspace $E^+$}

\subsubsection{Using the second moment inequality}		\label{s: dist into smi}

Denote $p=\e' m/4$, and let
\[
  Y_j=\e' m \cdot \dist^{-2} (B'_j, H'_j) \quad \text{for } j \in [m].
\]
By Lemma \ref{lem: distance}, for any $s>0$,
\[
 \Pr{Y_j>s} \le \left( \frac{C}{s} \right)^{2p}.
\]
Using Fubini's theorem, we conclude that
\[
 \E Y_j^p \le 1+ p \int_1^{\infty} s^{p-1} \cdot \P(Y_j>s) \, ds
 \le 1+\bar{C}^p,
\]
so $\norm{Y_j}_p \le C$. This is another instance where the assumption of the bounded density of the entries leads to a simplification of the proof. For a general distribution of entries, the event $\dist(B'_j, H'_j)=0$ may have a positive probability, and so $\norm{Y_j}_p$ may be infinite.

The bound on $\norm{Y_j}_p$ yields $\norm{ \sum_{j=1}^m Y_j}_p \le Cm$. Applying Markov's inequality, we get
\begin{align*}
  \P \left( \sum_{j=1}^m \dist^{-2} (B_j', H_j') \ge \frac{1}{\e' t}   \right)
  &=\P \left( \sum_{j=1}^m Y_j \ge \frac{m}{ t}   \right) \\
  &\le \frac{\E (\sum_{j=1}^m Y_j)^p}{(m/t)^p}
  \le (C t)^p
\end{align*}
for any $t>0$.

This estimate for $t=\tau^2$ combined with inequality \eqref{eq: second negative moment}
shows that
the event
\begin{equation}         \label{eq: sum moments small}
\mathcal{E}_1 := \left\{ \sum_{i=1}^m s_i(B)^{-2} \le \frac{1}{\tau^2 \e'} \right\}
\end{equation}
is likely: $\P((\mathcal{E}_1)^c) \le (C' \tau)^{\e' m/2}$.

\subsubsection{A large subspace $E^+$ on which $B$ is bounded below}		\label{s: E+}

Fix a parameter $\tau>0$ for now, and assume that the event \eqref{eq: sum moments small} occurs.
By Markov's inequality, for any $\d >0$ we have
$$
\Big| \big\{ i: \; s_i(B) \le \d \sqrt{m} \big\} \Big|
= \Big| \big\{ i: \; s_i(B)^{-2} \ge \frac{1}{\d^2m} \big\} \Big|
\le \frac{\d^2 m}{\tau^2 \e'}.
$$
 Setting $\d =  \tau \e'/10$, we have
\begin{equation}         \label{eq: small singular values}
\Big| \big\{ i: \; s_i(B) \le  \frac{\tau \e'}{10} \sqrt{n} \big\} \Big| \le \frac{ \e' m}{100}.
\end{equation}
Let $v_i(B)$ be the right singular vectors of $B$, and
consider the (random) orthogonal decomposition
$\C^n = E^- \oplus E^+$,
where
\begin{align*}
E^- &= \Span\{v_i(B): \; s_i(B) \le \frac{ \tau \e'}{10} \sqrt{m}\}, \\
E^+ &= \Span\{v_i(B): \; s_i(B) > \frac{ \tau \e'}{10} \sqrt{m}\}.
\end{align*}
Inequality \eqref{eq: small singular values} means that $\dim_{\C}(E^-) \le \frac{ \e' m}{100}$.

Let us summarize. Recall that $\e' m=\e n/2$ and set $\t=(\e s)^2$ for some $s \in (0,1)$.
We proved that the event
\[
\DD_{E^-} := \left\{ \dim(E^-) \le \frac{ \e' m}{100} \right\}
\]
satisfies
\begin{equation}							\label{eq: E- large}
\P((\DD_{E^-})^c) \le (C_2 \tau)^{\e' m}= (C_3 \e s)^{\e n},
\end{equation}
so $E^-$ is likely to be a small subspace and $E^+$ a large subspace.
The choice of $\t$ was made to create the factor $\e^{\e n}$ in the probability bound above ensuring that we can suppress the factor $\binom{n}{\e n}$ arising from the union bound.
Moreover, by definition, $B$ is nicely bounded below on $E^+$:
\begin{equation}							\label{eq: B bounded below}
\inf_{x \in S_{E^+}} \|Bx\|_2 \ge \frac{ \tau \e'}{10} \sqrt{m}
\ge \frac{ s^2 \e^3}{80} \sqrt{n}.
\end{equation}

\subsection{$G$ is bounded below on the small complementary subspace $E^-$}		\label{s: G below on E-}

The previous argument allowed us to handle the subspace $E_{+}$ whose dimension is only slightly lower than $m$. Yet, it provided no information about the behavior of the infimum of $\norm{Bx}_2$ over the unit vectors from the complementary subspace $E_{-}$. To get the lower bound for this infimum, we will use the submatrix $G$ we have put aside. Recall that although the space $E_{-}$ is random, it depends only on $B$, and thus is independent of $G$. Conditioning on the matrix $B$, we can regard this space as fixed.
Our task therefore, is to establish a lower bound on  $\norm{Gx}_2$ over the unit vectors from $E_{-}$. To this end, we can use the Lemma \ref{l: small subspace}. However, this lemma establishes the desired bound probability at least $1- \exp(-c' \e' m)$. This probability is insufficient for our purposes (remember, the probability for a fixed set $I \subset[n]$ is multiplied by $\binom{n}{\e n} \sim (e/\e)^{\e n}$.)

The probability bound is easy to improve in case of the bounded densities. Replacing the small ball probability estimate for a fixed vector used in the proof of Lemma \ref{l: small subspace} with Lemma \ref{lem: fixed vector}, we derive the following lemma.
\begin{lemma}[Lower bound on a subspace]		\label{lem: G bounded below on E}
  Let $M \ge 1$ and $\mu \in (0,1)$.
  Let $E$ be a fixed subspace of $\C^m$ of dimension at most $\e' m/100$.
  Then, for every $\rho >0$, we have
  \begin{equation}							\label{eq: G bounded below on E}
  \Pr{\inf_{x \in S_E} \|Gx\|_2 < \rho \sqrt{\e' m} \text{ and } \BB_{G,M}}
  \le \left(   \frac{C M \rho^{0.98}}{\e'^{0.01}}  \right)^{\e' m}.
  \end{equation}
\end{lemma}
The proof of this lemma follows the same lines as that of Lemma \ref{l: small subspace} and is left to a reader.

Lemma \ref{lem: G bounded below on E} provides the desired bound for the space $E_{-}$.
Recall that $m=(1-\e)n$ and $\e'=\e/2(1-\e)$.
 Namely, if the events $\BB_{G,M}$ and $\DD_{E_{-}}$ occur, then the event
\[
  \LL_{E_{-}}:= \left \{ \inf_{x \in S^{m-1} \cap E_{-}} \norm{Gx}_2 \ge \rho \sqrt{\e' m} \right \}
\]
holds with probability at least $1-  \left(   \frac{C M \rho^{0.98}}{\e'^{0.01}}  \right)^{\e' m}$. This is already sufficient since choosing a sufficiently small $\rho$, say $\rho=(s\e')^3$ with any $s \in (0,1)$, we see that
\[
  \P(\LL_{E_{-}}^c) \le (CM s^3 \e^{2.9})^{\e n/2},
 \]
 so again we can suppress the factor $\binom{n}{\e n}$ arising from the union bound.

\subsection{Extending invertibility from subspaces to the whole space.}

Assume that the events $\DD_{E_{-}}$ and $\LL_{E_{-}}$ occur. We know that if $\BB_{A,M}$ occurs, then this is likely:
\[
 \P( \BB_{A,M} \cap \DD_{E_{-}} \cap \LL_{E_{-}}) \ge \P(\BB_{A,M})- (Cs)^{\e n}.
\]

Under this assumption, we have uniform lower bounds on $\norm{A x}_2$ on the unit speres of both $E_{+}$ and $E_{-}$. The extension of these bounds to the whole unit sphere of $\C^m$ is now deterministic. It relies on the following lemma from linear algebra.

\begin{lemma}[Decomposition]			\label{lem: A decomposition}
  Let $A$ be an $m \times n$ matrix. Let us decompose $A$ as
  \[
  A =
  \begin{bmatrix}
  \phantom{X} B \phantom{X} \\ \phantom{X} G \phantom{X}
  \end{bmatrix},
  \quad B \in \C^{m_1 \times n}, \; G \in \C^{m_2 \times n}, \; m=m_1+m_2.
  \]
  Consider the orthogonal decomposition
  $\C^n = E^- \oplus E^+$
  where $E^-$ and $E^+$ are eigenspaces\footnote{In other words, $E^-$ and $E^+$
  are the spans of two disjoint subsets of right singular vectors of $B$.}
   of $B^* B$. Denote
  \[
  s_A = s_{\min} (A), \; s_B = s_{\min}(B|_{E^+}), \; s_G = s_{\min}(G|_{E^-}).
  \]
  Then
  \begin{equation}         \label{eq: A decomposition}
  s_A \ge \frac{s_B s_G}{4 \|A\|}.
  \end{equation}
\end{lemma}

\begin{proof}
Let $x \in S^{n-1}$. We consider the orthogonal decomposition
\[
x = x^- + x^+, \quad x^- \in E^-, \, x^+ \in E^+.
\]
We can also decompose $Ax$ as
\[
\|Ax\|_2^2 = \|Bx\|_2^2 + \|Gx\|_2^2.
\]
Let us fix a parameter $\theta \in (0,1/2)$ and consider two cases.

\smallskip

{\em Case 1: $\|x^+\|_2 \ge \theta$.} Then
\[
\|Ax\|_2 \ge \|Bx\|_2 \ge \|Bx^+\|_2 \ge s_B \cdot \theta.
\]

{\em Case 2: $\|x^+\|_2 < \theta$.}
In this case, $\|x^-\|_2 = \sqrt{1-\|x^+\|_2^2} \ge 1/2$.
Thus
\begin{align*}
\|Ax\|_2
  &\ge \|Gx\|_2 \ge \|Gx^-\|_2 - \|Gx^+\|_2 \\
  &\ge \|Gx^-\|_2 - \|G\| \cdot \|x^+\|_2
  \ge s_G \cdot \frac{1}{2} - \|G\| \cdot \theta.
\end{align*}
Using that $\|G\| \le \|A\|$, we conclude that
\[
s_A = \inf_{x \in S^{n-1}} \|Ax\|_2 \ge \min \Big( s_B \cdot \theta, \; s_G \cdot \frac{1}{2} - \|A\| \cdot \theta \Big).
\]
Optimizing the parameter $\theta$, we conclude that
\[
s_A \ge \frac{s_B s_G}{2 (s_B + \|A\|)}.
\]
Using that $s_B$ is bounded by $\|A\|$, we complete the proof.
\end{proof}
Combining Lemma \ref{lem: A decomposition} with the previously obtained bounds \eqref{eq: B bounded below} and \eqref{eq: G bounded below on E}, we complete the proof of Proposition \ref{prop: reduction}, and thus, the no-gaps delocalization Theorem \ref{thm: delocalization continuous}.

\section{Applications of the no-gaps delocalization}

\subsection{Erd\H{o}s-R\'enyi graphs and their adjacency matrices}
In this section we consider two applications of the no-gaps delocalization to the spectral properties of the Erd\H{o}s-R\'enyi random graphs. Let $p \in (0,1)$. Consider a graph $G=(V,E)$ with $n$ vertices such that any pair of vertices is connected by an edge with probability $p$, and these events are independent for different edges. This model of a random graph is called an Erd\H{o}s-R\`enyi or $G(n,p)$ graph. Let $A_G$ be the adjacency matrix matrix of a graph $G$, i.e., the matrix of zeros and ones with $1$ appearing on the spot $(i,j)$ whenever the vertices $i$ and $j$ are connected.
 We will need several standard facts about the Erd\H{o}s-R\'enyi graphs listed in the followiing proposition.
 \begin{proposition}  \label{prop: G(n,p)}
 Let $p \ge C \frac{\log n}{n}$ for some $C>1$.
 Let $G(V,E)$ be a $G(n,p)$ graph. Then $G$ has the following properties with probability $1-o(1)$.
 \begin{enumerate}
   \item Let $R \subset V$ be an independent set, i.e., no two vertices from $R$ are connected by an edge. Then
       \[
        |R| \le C \frac{\log n}{p}.
       \]

   \item Let $P,Q \subset V$ be disjoint sets of vertices with
   \[
    |P|, |Q| \ge C \frac{\log n}{p}.
   \]
   Then  there is an edge connecting a vertex from $P$ and a vertex from $Q$.
   \item The degree of any vertex $v \in V$ is close to its expectation:
   \[
    np-\log n \cdot \sqrt{np} \le d_v \le np + \log n \cdot \sqrt{np}
   \]
   \item Let $\hat{\l}_1 \ge \etc \ge \hat{\l}_n$ be eigenvalues of the normalized adjacency matrix $\hat{A}:=D_G^{-1/2} A_G D_G^{-1/2}$ where $D_G$ is the diagonal matrix $D_G=\text{\rm diag}(d_v, \ v \in V)$. Then
       \[
       \hat{\l}_1=1, \quad \text{and} \quad |\hat{\l}_j| \le \frac{C}{\sqrt{np}}
       \text{ for } j \ge 1.
       \]
   \item For every subset of vertices $J \subset V$, let $\textrm{Non-edges}(J)$ be the set of all pairs of vertices $v,w \in J$ which are not connected by an edge.
       Then
\[
 (1-p) \binom{|J|}{2} - n^{3/2} \le |\text{Non-edges}(J)| \le (1-p) \binom{|J|}{2} + n^{3/2}.
 \]
 \end{enumerate}
 \end{proposition}
 We leave the proof of these properties to a reader.

Considering the vector of all ones, we realize that $\norm{A_G} = \Omega(np)$ with high probability.
Hence, when  $p$ is fixed, and $n \to \infty$, this makes the event $\BB_{A_G,M}$ unlikely. However, Remark \ref{rem: shift} shows that we can replace this event by the event $\BB_{A_G- p \mathbf{1}_n,M}$ which holds with probability close to $1$.
Indeed,
\[
  A_G- p \mathbf{1}_n=B-\D,
\]
where $B$ is a symmetric random matrix with centered Bernoulli$(p)$ entries which are independent on and above the diagonal, and $\D$ is the diagonal matrix with i.i.d. Bernoulli$(p)$ entries. Here, $\norm{\D} \le 1$, and by a simple $\e$-net argument, $\norm{B} \le C\sqrt{np}$ with probability close to $1$. This decomposition is reflected in the structure of the spectrum of $A_G$. Let us arrange the eigenvalues of $A_G$ in the decreasing order: $\l_1(G) \ge \etc \ge \l_n(G)$. Then with high probability, $\l_1(G)=\Omega(np)$ and $|\l_j(G)| =O(\sqrt{np})$, where the last equality follows from $\norm{A_G- p \mathbf{1}_n} = O(\sqrt{np})$ and the interlacing property of the eigenvalues.

 Remark \ref{rem: shift} shows that no-gaps delocalization can be extended to the matrix $A_G$ as well.   We will use this result in combination with the $\ell_{\infty}$ delocalization which was established for the $G(n,p)$ graphs by Erd\H{o}s et. al. \cite{EKYY}. They proved that with probability at  least $1-\exp(-c \log^2 n)$, any  unit eigenvector $x$ of $A_G$ satisfies
\begin{equation} \label{eq: l_infty}
 \norm{x}_{\infty} \le \frac{\log^C n}{\sqrt{n}}.
\end{equation}

\subsection{Nodal domains of the eigenvectors of the adjacency matrix}
Let $f$ be an eigenfunction of a self-adjoint linear operator. Define the (strong) nodal domains of $f$ as connected components of the sets where $f$ is positive or negative. Nodal domains of the Laplacian on a compact smooth manifold is a classical object in analysis. If the eigenvalues are arranged in the increasing order, the number of nodal domains of the eigenfunction corresponding to the $k$-th eigenvalue does not exceed $k$ and tends to infinity as $k \to \infty$.

If we consider a finite-dimensional setup, the eigenfunctions of self-adjoint linear operators are replaced by the eigenvectors of symmetric matrices. In 2008, Dekel, Lee, and Linial \cite{DLL} discovered that the nodal domains of the adjacency matrices of $G(n,p)$ graphs behave strikingly different from the eigenfunctions of the Laplacian on a manifold. Namely, they proved that with high probability, the number of nodal domains of any non-first eigenvector of a $G(n,p)$ graph is bounded by a constant depending only on $p$. Later, their result was improved by Arora and Bhaskara \cite{AB}, who showed that with high probability, the number of nodal domains is $2$ for all non-first eigenvectors. Also, Nguyen, Tao, and Vu \cite{NTV} showed that the eigenvector of a $G(n,p)$ graph cannot have zero coordinates with probability close to $1$. These two results in combination mean that for each non-first eigenvector, the set of vertices of a $G(n,p)$ graph splits into the set of positive and negative coordinates both of which are connected.

Let us derive Dekel-Lee-Linial-Arora-Bhaskara theorem from the delocalization properties of an eigenvector. Assume that $p$ is fixed to make the presentation easier.  Let $x \in S^{n-1}$ be a non-first eigenvector of $A_G$, and denote its coordinates by $x_v, \ v \in V$. Let $P$ and $N$ be the largest nodal domains of positive and negative and negative coordinates. Since $x$ is orthogonal to the first eigenvector having all positive coordinates, both $P$ and $N$ are non-empty. Denote $W=V \setminus (P \cup N)$. Our aim is to prove that with high probability, $W= \varnothing$. We start with proving a weaker statement that the cardinality of $W$ is small.
\begin{proposition} \label{prop: W is small}
\[
  |W| \le C \frac{\log n^2}{p^2}
\]
 with probability  $1-o(1)$.
\end{proposition}
\begin{proof}

 Pick a vertex from each positive nodal domain. These vertices cannot be connected by edges as they belong to different connected components. Using Proposition \ref{prop: G(n,p)} (1), we derive that, with high probability, the number of such domains does not exceed $C \frac{\log n}{p}$. The same bound holds for the number of negative nodal domains.

Consider a nodal domain $W_0 \subset W$ and assume that $|W_0| \ge C \frac{\log n}{p}$. If this domain is positive, $|P| \ge C \frac{\log n}{p}$ as well, since $P$ is the largest nodal domain. This contradicts Proposition \ref{prop: G(n,p)} (2) as two nodal domains of the same sign cannot be connected. Combining this with the previous argument, we complete the proof of the proposition.
\end{proof}

Now, we are ready to prove that $W= \varnothing$ with probability $1-o(1)$. Assume to the contrary that there is a vertex $v \in W$, and assume that $x_v<0$. Let $\Gamma(v)$ be the set of its neighbors in $G$. Then $\Gamma(v) \cap N= \varnothing$ as otherwise $v$ would be an element of $N$. Since $x$ is an eigenvector,
\[
 \l x_v = \sum_{u \in \Gamma(v)} x_u
 = \sum_{u \in \Gamma(v)\cap P} x_u  + \sum_{u \in \Gamma(v) \cap W} x_u.
\]
Here $|\l| \le \sqrt{np}$ because $\l$ is a non-first eigenvalue. Then
\begin{align*}
 \norm{x|_{\Gamma(v)}}_1
 &\le \sum_{u \in \Gamma(v)\cap P} x_u  + \sum_{u \in \Gamma(v) \cap W}|x_u|
 \le 2 \sum_{u \in \Gamma(v) \cap W}|x_u| + |\l| \cdot |x_v| \\
 &\le \left( 2 |\Gamma(v) \cap W|+|\l| \right) \cdot \norm{x}_{\infty}.
\end{align*}
By Proposition \ref{prop: W is small} and \eqref{eq: l_infty}, this quantity does not exceed $\log^C n$. Applying \eqref{eq: l_infty} another time, we conclude that
\[
  \norm{x|_{\Gamma(v)}}_2 \le \sqrt{\norm{x|_{\Gamma(v)}}_1 \cdot \norm{x}_{\infty}}
  \le n^{-1/4} \log^C n.
\]
In combination with Proposition \ref{prop: G(n,p)} (3), this shows that a large set $\Gamma(v)$ carries a small mass, which contradicts the no-gaps delocalization. This completes the proof of  Dekel-Lee-Linial-Arora-Bhaskara theorem.

The same argument shows that with high probability, any vertex of the positive nodal domain is connected to the negative domain and vice versa. More precisely, we have the following stronger statement.

\begin{lemma} 
Let $p \in (0,1)$.
Let $x \in S^{n-1}$ be a non-first eigenvector of $A_G$. Let $V=P \cup N$ be the decomposition of $V$ into the positive and negative nodal domains corresponding to $x$. Then with probability greater than $1-\exp(-c' \log^2 n)$, any vertex in $P$ has at least $\frac{c n}{\log^C n}$ neighbors in $N$, and  any vertex in $N$ has at least $\frac{ n}{\log^C n}$ neighbors in $P$.
\end{lemma}

\begin{proof}
Since $\l$ is a non-first eigenvalue, $|\l| \le c \sqrt{n}$ with high probability. Assume that the vector $x$ is delocalized in both $\ell_{\infty}$ and no-gaps sense.
Let $w \in P$, and assume that 
 \[
  |\Gamma(w) \cap N| \le \frac{n}{\log^{4C} n},
 \]
 where $\Gamma(w)$ denotes the set of neighbors of $w$. We have
\[
 \l x_w= \sum_{v \in \Gamma(w) \cap P} x_v+ \sum_{v \in \Gamma(w) \cap N} x_v,
\]
and as before,
\begin{align*}
 \norm{x|_{\Gamma(w)}}_1 
 &= \sum_{v \in \Gamma(w) \cap P} x_v+ \sum_{v \in \Gamma \cap N} |x_v|
 \le 2 \sum_{v \in \Gamma(w) \cap N} |x_v|+|\l| \cdot |x_w| \\
 &\le 2\frac{n}{\log^{4C} n} \cdot \frac{\log^C n}{\sqrt{n}}+ c \sqrt{n} \cdot \frac{\log^{4C} n}{\sqrt{n}}.
\end{align*}
 Hence,
\[
 \norm{x|_{\Gamma(w)}}_2 \le \sqrt{\norm{x}_{\infty} \cdot \norm{x|_{\Gamma(w)}}_1}
 \le \frac{\sqrt{2}}{\log^C n}, 
\]
which contradicts the no-gaps delocalization, as $|\Gamma(w)|\ge c np$ with high probability.
The proof finishes by application of the union bound over $w$.
\end{proof}

\subsection{Spectral gap of the normalized Laplacian and Braess's paradox}

In some cases, the addition of a new highway to an existing highway system may increase the traffic congestion. This phenomenon discovered in 1968 by Braess became known as Braess's paradox. Since its discovery, a number of mathematical models have been suggested to explain this paradox. We will consider one such model suggested by  Chung et. al. \cite{CYZ}.

We will model the highway system by an Erd\H{o}s-R\`enyi graph $G(n,p)$. The congestion of the graph will be measured in terms of its \emph{normalized Laplacian} which we will define in a moment. Let $A_G$ be the adjacency matrix of the graph $G$, and let $D_G=(d_v, \ v \in V)$ be $n \times n$ the diagonal matrix whose diagonal entries are the degrees of the vertices. The normalized Laplacian of $G$ is defined as
\[
 \LL_G:=I_n-D_G^{-1/2} A_G D_G^{-1/2}.
\]
The normalized Laplacian is a positive semidefinite matrix, so it has a real non-negative spectrum. We will arrange it in the increasing order: $0 = \l_1(\LL_G) \le \ldots \le \l_n(\LL_G)$. The eigenvalue $\l_1(\LL_G)=0$ corresponds to the eigenvector $Y$, whose coordinates are $Y_v=d_v^{1/2}, \ v \in V$.
The quantity $\l_2(\LL_G)$ is called \emph{the spectral gap} of $G$.The spectral gap appears in the Poincare inequality, so it is instrumental in establishing measure concentration properties of various functionals. Also, the reciprocal of the spectral gap defines the relaxation time for a random walk on a graph. In this quality, it can be used to measure the congestion of the graph considered as a traffic network: the smaller spectral gap corresponds to a bigger congestion.

For a  graph $G$, and let $a_{-}(G)$ be the fraction of non-edges $(u,v) \notin E$ such that the addition of $(u,v)$ to the set of edges decreases the spectral gap. Intuitively, the addition of an edge should increase the spectral gap as it brings the graph closer to the complete one, for which the spectral gap is maximal. However, the numerical experiments showed that the addition of an edge to a random graph frequently yields an opposite effect. This numerical data led to the following conjecture, which is a variant of the original conjecture of Chung.
\begin{conjecture}  \label{conj: gap decrease}
 Let $p \in (0,1)$ be fixed. Then there exists a constant $c(p)$ such that
 \[
   \lim_{n \to \infty} \Pr{a_{-}(G) \ge c(p) } =1.
 \]
\end{conjecture}
This conjecture has been proved by Eldan, R\`{a}sz, and Shramm \cite{ERS}. Their proof is based on the following deterministic condition on the eigenvectors which ensures that the spectral gap decreases after adding an edge.
\begin{proposition} \label{prop: gap decrease}
 Let $G$ be a graph such that $(1/2) np \le d_v \le (3/2)np$ for all vertices $v \in V$. Let $x \in S^{n-1}$  be the eigenvector of $\LL_G$ corresponding to $\l_2(G)$.
 If $(u,w) \notin E$ is a non-edge, and
 \[
  \frac{1}{\sqrt{np}} \left(x_u^2+x_w^2\right) +c_1(np)^{-2} < c_2 x_u x_v,
 \]
 then the addition of the edge $(u,w)$ to $G$ decreases the spectral gap.
\end{proposition}
The proof of proposition \ref{prop: gap decrease} requires a tedious, although a rather straightforward calculation. Denote by $y \in S^{n-1}$ the \emph{first} eigenvector of the graph $G_{+}$ obtained from $G$ by adding the edge $(u,w)$, and let $Q: \R^n \to \R^n$ be the orthogonal projection on the space $y^{\perp}$. By the variational definition of the second eigenvalue,
\[
 \l_2(G_{+})= \inf_{z \in y^{\perp} \setminus \{0\} }  \frac{\pr{z}{\LL_{G_{+}} z}}{\norm{z}_2^2}
 \le \frac{\pr{Qx}{\LL_{G_{+}} Qx}}{\norm{Qx}_2^2}
 =\frac{\pr{x}{\LL_{G_{+}} x}}{1-\pr{x}{y}^2},
\]
where the last equality follows since $\LL_{G_{+}} y=0$.
In the last formula, $y=\D/ \norm{\D}_2$, where $\D$ is the vector with coordinates $\D_v= \sqrt{d_v}$ for $v \notin \{u,w\}$ and $\D_v=\sqrt{d_v+1}$ for $v \in \{u,w\}$. The matrix $\LL_{G_{+}}$ can be represented in a similar way:
\[
 \LL_{G_{+}}=I_n-D_{+}^{-1/2} A_{G_+} D_{G_{+}}^{-1/2},
\]
where $A_G+(e_ue_w^T+e_we_u^T)$ and $D_{G_{+}}$  is defined as $D_G$ above.
The proposition follows by substituting these formulas in the previous estimate of $\l_2(G_{+})$ and simplifying the resulting expression. A reader can find the detailed calculation in \cite{ERS}.

Proposition \ref{prop: gap decrease}  allows us to lower bound $a_{-}(G)$.
The main technical tool in obtaining such a bound is delocalization. We will need both the $\ell_{\infty}$ and no-gaps delocalization of the second eigenvector of $\LL_G$. Both properties hold for the eigenvectors of $A_G$, so our task is to extend them to the normalized Laplacian.
\begin{lemma} \label{lem: Laplacian}
 Let $p \in (0,1)$. Let $f \in S^{n-1}$ be the second eigenvector of $\LL_G$. Then  with probability at least $1-\exp(-c \log^2 n)$,
 \[
   \norm{f}_{\infty} \le n^{-1/4} \log^C n
 \]
 and there exists a set $W \subset V$ with $|W^c| \le c' n^{1-1/48}$ such that for any $v \in W$,
 \[
   |f_v| \ge n^{-5/8}
 \]
 Here, $C,c,c'$ are positive constants whose value may depend on $p$.
\end{lemma}

\begin{proof}
Let us start with the $\ell_{\infty}$ delocalization.
 Let $d=np$ be the expected degree of a vertex, and set
 \[
  x=d^{1/2} D_G^{-1/2} f.
 \]
 By Proposition \ref{prop: G(n,p)} (3), $d^{1/2} D_G^{-1/2}= \text{diag}(s_v, \ v \in V)$, where  $s_v=1+o(1)$ for all $v \in V$, and  $\norm{x}_2=1+o(1)$ with probability close to $1$.
Hence, it is enough to bound $\norm{x}_{\infty}$. Let us check that $x$ is an approximate eigenvector of  $A_G$ corresponding to the approximate eigenvalue $\hat{\l}_2 d$, where $\hat{\l}_2$ is the second eigenvalue of the normalized adjacency matrix $D_G^{-1/2}A_G D_G^{-1/2}$.
By Proposition \ref{prop: G(n,p)} (4), $\hat{\l}_2 \le c/\sqrt{np}$ with high probability, hence
\begin{align*}
 \norm{A_G D_G^{-1/2}f - \hat{\l}_2 d D_G^{-1/2} f}_2
 &=|\hat{\l}_2| \cdot \norm{D_G^{1/2}f -  d D_G^{-1/2} f}_2  \\
 &\le \frac{c}{\sqrt{n}} \cdot  \max_{v \in V} d_v^{-1/2} \cdot \max_{v \in V} |d_v-d| \\
 &\le \frac{c}{n} \cdot \max_{v \in V} |d_v-d|
  \le \frac{C \log n}{\sqrt{n}},
\end{align*}
and so
\begin{equation}\label{eq: approx eigenvector}
  \norm{A_G x- \hat{\l}_2 d x}_2 \le C \log n.
\end{equation}
 Let $\rho \ge 1$. By the local semicircle law for $A_G$ (\cite{EKYY}, Theorem 2.10), any interval $[b,b+\rho]$ contains at most
\[
  N(\rho):=c \rho \sqrt{n}
\]
eigenvalues of $A_G$ with probability greater than $1-\exp(-c \log^2 n)$.

Denote the eigenvalues of $A_G$ by $\mu_1 \etc \mu_n$ and the corresponding eigenvectors by  $u_1 \etc u_n \in S^{n-1}$, and let $\a_j=\pr{x}{u_j}$.
Set $\mu=\hat{\l}_2 d$ and let $P_\t$ be the orthogonal projection on the span of the eigenvectors corresponding to the eigenvalues of $A_G$ in the interval $[\mu-\t, \mu+\t]$.
Then
\begin{align*}
 \t \norm{(I-P_\t)x}_2 
 &= \t \left( \sum_{|\mu_j -\mu|>\t} \a_j^2 \right)^{1/2} 
 \le  \left( \sum_{|\mu_j -\mu|>\t} (\mu_j-\mu)^2 \a_j^2 \right)^{1/2} \\
 &\le \norm{(A_G-\mu)x}_2 \le C \log n.
\end{align*}
and so,
\begin{equation}  \label{eq: spectr project}
 \norm{(I-P_\t)x}_2  \le \left( C \frac{\log n}{\t} \wedge 1 \right).
\end{equation}

For any $\t \ge 0$ and any $\rho \ge 1$,
\begin{align*}
 \norm{(P_{\t+ \rho}-P_\t)x}_{\infty}
 &=\norm{\sum_{|\mu_j-\mu| \in [\t, \t+\rho]} \a_j u_j}_{\infty}
 =\max_{v \in V} \left| \sum_{|\mu_j-\mu| \in [\t, \t+\rho]} \a_j u_{j,v} \right| \\
 &\le \left(\sum_{|\mu_j-\mu| \in [\t, \t+\rho]} \a_j^2 \right)^{1/2} \cdot
   \max_{v \in V} \left(\sum_{|\mu_j-\mu| \in [\t, \t+\rho]}  u_{j,v}^2 \right)^{1/2} \\
 &\le \norm{(P_{\t+\rho}-P_\t)x}_2 \cdot N^{1/2}(\rho) \cdot \max_{j \in [n]}\norm{u_j}_{\infty} \\
 &\le \norm{(I-P_\t)x}_2 \cdot \sqrt{\rho n^{1/2}} \cdot \frac{\log^C n}{\sqrt{n}},
\end{align*}
where we used \eqref{eq: l_infty} in the last inequality.
Combining this with \eqref{eq: spectr project}, we get
\[
  \norm{(P_{\t+ \rho}-P_\t)x}_{\infty}
  \le C \sqrt{\rho} \frac{\log^C n}{n^{1/4}} \cdot \left(\t^{-1} \wedge 1 \right).
\]
Applying this inequality with $\t=\rho=2^k, \ k \in \{0\} \cup \N$, we derive the required norm bound:
\begin{align*}
 \norm{x}_{\infty}
 &\le \norm{P_1 x}_{\infty}+ \sum_{k=0}^{\infty} \norm{(P_{2^{k+1}}-P_{2^k})x}_{\infty} \\
 &\le C  \frac{\log^C n}{n^{1/4}}+ \sum_{k=1}^{\infty} C 2^{-k/2} \frac{\log^C n}{n^{1/4}}
 \le C n^{-1/4} \log^C n.
\end{align*}
By the discussion above, $\norm{f}_{\infty} \le 2\norm{x}_{\infty}$ which finishes the proof of the first part of the lemma.

Now, let us prove the lower bound on the absolute values of most of the coordinates of $f$. As before, it is enough to prove a similar  bound on the coordinates of $x$. Assume to the contrary that there is a set $U \subset V$ with $|U| > cn^{1-1/48}$ such that for any $v \in U$, $|x_v| \le n^{-5/8}$.
Then
\[
 \norm{x_U}_2 \le \sqrt{n} \cdot n^{-5/8}=n^{-1/8}.
\]
Inequality \eqref{eq: approx eigenvector} shows that $x$ is an approximate eigenvector of $A_G$. Since $n^{-1/8} \gg Cn^{-1/2} \log^C n$, by Remarks \ref{rem: shift} and \ref{rem: approximate eigenvectors}, we can apply Theorem  \ref{thm: delocalization general} to $x$ with $s$ being an appropriately small constant and $\e=(1/s) n^{-1/48}$, so $(\e s)^6= n^{-1/8}$. This theorem shows that such set $U$ exists with probability at most $\exp(-\e n) \ll \exp(-c \log^2 n)$. The proof of the lemma is complete.
\end{proof}

Equipped with Proposition \ref{prop: gap decrease} and Lemma \ref{lem: Laplacian}, we can prove a stronger form of the conjecture showing that $c\ge 1/2-o(1)$. Let us formulate it as a theorem.
\begin{theorem}  \label{thm: gap decrease}
 Let $p \in (0,1)$, and let $G$ be a $G(n,p)$ graph. Then with probability $1-o(1)$,
 \[
   a_{-}(G) \ge \frac{1}{2}-O(n^{-c}).
 \]
\end{theorem}

\begin{proof}
Let $f \in S^{n-1}$ be the eigenvector of $\LL_G$ corresponding to the second eigenvalue, and assume that the event described in Lemma \ref{lem: Laplacian} occurs. Let $W$ be the set defined in this lemma.
 Set
 \[
  W_+=\{v \in W:  \ f_v>0\}, \quad \text{and} \quad W_{-}=\{v \in W:  \ f_v<0\}.
 \]
 For any $v,w \in W_{+}$,
 \[
  \frac{f_v^2+f_w^2}{f_v f_w} \le 2 \max_{v,w \in W_{+}} \frac{f_v}{f_w}
  \le C n^{3/8} \log^C n \ll \sqrt{n}.
 \]
 Hence, if $(v,w)$ is a non-edge, then Proposition \ref{prop: gap decrease} implies that adding it to $G$ decreases the spectral gap.
 Similarly, we can show that adding any non-edge whose vertices belong to $W_{-}$, decreases the spectral gap as well.
 Let us count the number of the non-edges in $W_{+}$ and $W_{-}$ and compare it to the  total number of the non-edges.
Using Property (5), and the bound $|W^c| \le c n^{1-1/48}$, we obtain
 \begin{align*}
   a_{-}(G)
   &\ge \frac{|\text{Non-edges}(W_{+})|+|\text{Non-edges}(W_{-})|}{|\text{Non-edges}(V)|} \\
   &\ge \frac{(1-p) \left[\binom{|W_{+}|}{2}+ \binom{|W_{-}|}{2} \right]-2 n^{-3/2}}{(1-p) \binom{n}{2} +n^{3/2}} \\
   &\ge \frac{(1-p) \left[ \left(\frac{|W_{+}|+|W_{-}|}{2} \right)^2 -|W_{+}|-|W_{-}| \right]-2n^{3/2}}{(1-p) \binom{n}{2} +n^{3/2}}
   \ge \frac{1}{2}-O(n^{-c}),
 \end{align*}
 as claimed.
\end{proof}

\bibspread

\end{document}